\documentclass[12pt,a4paper]{article}
\usepackage{amsmath}
\usepackage{amssymb}
\usepackage{amsthm}
\usepackage{graphicx}
\usepackage{tikz}
\usetikzlibrary{arrows.meta}
\usepackage{cite}
\usepackage{hyperref}

\makeatletter
  \def\@seccntformat#1{%
    \@nameuse{@seccnt@prefix@#1}%
    \@nameuse{the#1}%
    \@nameuse{@seccnt@postfix@#1}%
    \@nameuse{@seccnt@afterskip@#1}}
  \def\@seccnt@prefix@section{}
  \def\@seccnt@postfix@section{.}
  \def\@seccnt@afterskip@section{\hspace{.5em}}
  \def\@seccnt@prefix@subsection{}
  \def\@seccnt@postfix@subsection{.}
  \def\@seccnt@afterskip@subsection{\hspace{.5em}}
\makeatother

\makeatletter
\renewcommand\section{
  \@startsection{section}{3}{\z@}%
  {-3.25ex\@plus -1ex \@minus -.2ex}%
  {1.5ex \@plus .2ex}%
  {\normalfont\normalsize\bfseries\mathversion{bold}}}
\renewcommand\subsection{
  \@startsection{subsection}{3}{\z@}%
  {-3.25ex\@plus -1ex \@minus -.2ex}%
  {1.5ex \@plus .2ex}%
  {\normalfont\normalsize\bfseries\mathversion{bold}}}
\makeatother

\makeatletter \@addtoreset{equation}{section} \makeatother
\renewcommand{\theequation}{\arabic{section}.\arabic{equation}}

\allowdisplaybreaks[1]

\let\oldthebibliography\thebibliography
\renewcommand\thebibliography[1]{
  \oldthebibliography{#1}\setlength{\itemsep}{0.4ex}}

\theoremstyle{plain}
\newtheorem{thm}{Theorem}[section]
\newtheorem{lem}[thm]{Lemma}
\newtheorem{prop}[thm]{Proposition}
\newtheorem{cor}[thm]{Corollary}

\theoremstyle{definition}
\newtheorem{algo}[thm]{Algorithm}

\newtheorem{dfn}[thm]{Definition}

\theoremstyle{remark}
\newtheorem{rem}[thm]{Remark}

\theoremstyle{plain}
\newtheorem*{thm*}{Theorem}
\newtheorem*{lem*}{Lemma}
\newtheorem*{prop*}{Proposition}
\newtheorem*{cor*}{Corollary}
\newtheorem*{conj*}{Conjecture}

\theoremstyle{definition}
\newtheorem*{algo*}{Algorithm}
\newtheorem*{ass*}{Assumption}
\newtheorem*{dfn*}{Definition}

\theoremstyle{remark}
\newtheorem*{rem*}{Remark}

\newcommand{\nn}{\nonumber}

\newcommand{\grp}[1]{\mathrm{#1}}

\newcommand{\bbR}{\mathbb{R}}
\newcommand{\bbC}{\mathbb{C}}
\newcommand{\bbZ}{\mathbb{Z}}
\newcommand{\bbH}{\mathbb{H}}
\newcommand{\varth}{\vartheta}

\newcommand{\cD}{\mathcal{D}}
\newcommand{\cS}{\mathcal{S}}
\newcommand{\cT}{\mathcal{T}}
\newcommand{\ha}{\hat{a}}
\newcommand{\hb}{\hat{b}}
\newcommand{\hc}{\hat{c}}
\newcommand{\hd}{\hat{d}}
\newcommand{\halpha}{\hat{\alpha}}
\newcommand{\hpsi}{\psi_R}
\newcommand{\tf}{\tilde{f}}
\newcommand{\tu}{\tilde{u}}
\newcommand{\tI}{\tilde{I}}
\newcommand{\tpsi}{\tilde{\psi}}
\newcommand{\RS}{R}
\newcommand{\Gtri}{{\mathcal{G}}}
\newcommand{\Rtri}{R^\Gtri}
\newcommand{\Rwtri}{R^{\Gtri,\mathrm{w}}}
\newcommand{\RD}{R^{W(D_4)}}
\newcommand{\Rcov}{R^{\mathrm{cov}}}
\newcommand{\Sop}{\mathcal{I}}
\newcommand{\Rab}{\mathcal{R}}
\newcommand{\Rcd}{\tilde{\mathcal{R}}}
\newcommand{\vecal}{{\boldsymbol{\alpha}}}
\newcommand{\vece}{{\boldsymbol{\mathrm{e}}}}
\newcommand{\vecz}{{\boldsymbol{z}}}
\newcommand{\mapsfrom}{\mathrel{\reflectbox{\ensuremath{\mapsto}}}}

\newcommand{\gen}[4]{{\phi^{(#1,#2)}_{#3,#4}}}
\newcommand{\tv}[3]{\bigl(#1,#2\bigr)_{#3}}

\begin{document}


\def\papertitlepage{\baselineskip 3.5ex \thispagestyle{empty}}
\def\preprinumber#1#2{\hfill
\begin{minipage}{1.22in}
#1 \par\noindent #2
\end{minipage}}

%
\papertitlepage
\setcounter{page}{0}
\preprinumber{arXiv:2504.00546}{}
\vskip 2ex
\vfill
\begin{center}
{\large\bf\mathversion{bold}
The ring of $D_4$ triality invariants
}
\end{center}
\vfill
\baselineskip=3.5ex
\begin{center}
Kazuhiro Sakai\\

{\small
\vskip 6ex
{\it Institute for Mathematical Informatics, Meiji Gakuin University\\
1518 Kamikurata-cho, Totsuka-ku,
Yokohama 244-8539, Japan}\\
\vskip 1ex
{\tt kzhrsakai@gmail.com}

}
\end{center}
\vfill
\baselineskip=3.5ex
\begin{center} {\bf Abstract} \end{center}

$D_4$ triality invariants are modular forms
as well as invariant polynomials for a fiber product
of the modular group and the Weyl group of type $F_4$.
We show that the ring of $D_4$ triality invariants
satisfying a certain cusp condition
is isomorphic to the ring of joint covariants of
a binary cubic and a binary quadratic form.

\vfill
\vspace{5ex}

\noindent
2020 Mathematics Subject Classification: 11F11, 17B22, 13A50


\setcounter{page}{0}
\newpage
\renewcommand{\thefootnote}{\arabic{footnote}}
\setcounter{footnote}{0}
\setcounter{section}{0}
\baselineskip = 3.5ex
\pagestyle{plain}

\section{Introduction}

Automorphic forms often exhibit a close
connection with classical invariant theory.
Let $V_n$ denote the $\bbC$-vector space of binary forms of degree $n$,
on which $\grp{SL}_2(\bbC)$ acts linearly.
The ring of modular forms (without characters)
for the even unimodular lattice of signature $(2,18)$
is isomorphic to the invariant ring
$\bbC[V_{12}\oplus V_8]^{\grp{SL}_2(\bbC)}$
\cite{Odaka:2018,Nagano:2021}.
The ring of Weyl invariant $E_8$
weak Jacobi forms is isomorphic to the covariant ring
$\bbC[V_6\oplus V_4\oplus\bbC^2]^{\grp{SL}_2(\bbC)}$
\cite{Sakai:2024vby}.
In this paper we show a similar relation
in a more elementary context:
\begin{thm}[Theorem~\ref{thm:RtriandCov}]\label{thm:main}
The ring of $D_4$ triality invariants is isomorphic to
the covariant ring
$\bbC[V_3\oplus V_2\oplus\bbC^2]^{\grp{SL}_2(\bbC)}$.
\end{thm}

We introduce the notion of $D_4$ triality invariants as follows.
For any irreducible root system $\RS$, 
the associated Weyl group $W(\RS)$ acts
on the root lattice $\Lambda_\RS$
as an automorphism group.
In some cases the full automorphism group of $\Lambda_\RS$ is
larger than $W(\RS)$.
In particular, the full automorphism group of $\Lambda_{D_4}$
is $W(F_4)$.
The quotient $W(F_4)/W(D_4)$ is isomorphic to
the symmetric group $S_3$ and is called the triality,
which permutes the three legs of the $D_4$ Dynkin diagram.
Let $\phi_1,\phi_2$ be group homomorphisms
given by the quotient maps
\begin{align}
\begin{aligned}
\phi_1&:\grp{SL}_2(\bbZ)\to\grp{SL}_2(\bbZ)/\Gamma(2)\cong S_3,\\
\phi_2&:W(F_4)\to W(F_4)/W(D_4)\cong S_3.
\end{aligned}
\end{align}
Identifying the two $S_3$'s we can define the fiber product
\begin{align}
\Gtri:=\grp{SL}_2(\bbZ)\times_{S_3}W(F_4)
 =\{(g,w)\in\grp{SL}_2(\bbZ)\times W(F_4)\,|\,\phi_1(g)=\phi_2(w)\}.
\end{align}
A $D_4$ triality invariant is then defined as
a holomorphic function $\varphi:\bbH\times\bbC^4\to\bbC$
that is a modular form as well as an invariant polynomial
for $\Gtri$
and satisfies the cusp condition
$\lim_{q\to 0}|\varphi(\tau,z_1,z_2,z_3,q^{-1/2}\Lambda)|<\infty$
with $q=e^{2\pi i\tau},\ 0<|\Lambda|<\infty$.
$D_4$ triality invariants form a bigraded algebra $\Rtri_{*,*}$
over the ring of modular forms, graded by modular weight and
polynomial degree with respect to $z_i$.

The group $\Gtri$ arises naturally as a symmetry of
the $\mathcal{N}=2$ supersymmetric $\grp{SU}(2)$ gauge theory
with four flavors.
This theory, often called the $N_\mathrm{f}=4$ theory, is one of
the most well-known gauge theory models in mathematical physics
\cite{Seiberg:1994aj,Alday:2009aq}.
Remarkably, the spectrum of this theory is invariant under 
the permutations of electrons, monopoles and dyons
which induce simultaneous actions of
the $\grp{SL}_2(\bbZ)/\Gamma(2)$ electric-magnetic duality and
the $S_3$ outer automorphisms of the $\grp{Spin}(8)$ flavor symmetry.

The cusp condition also has a natural origin in physics
\cite{Seiberg:1994aj}:
it is the condition for a physical quantity
to stay finite under the renormalization group flow
from the $N_\mathrm{f}=4$ theory to
its descendant theories with less flavors.
Thus, $D_4$ triality invariants
serve as fundamental building blocks
in describing physical quantities of the theory
and it is an important problem to clarify their ring structure.

In this paper we shall prove Theorem~\ref{thm:main}
in the following way.
We first construct an injection $\Sop$, which
simplifies checking the cusp condition:
a $\Gtri$-invariant $\varphi$ satisfies the condition
if and only if $\Sop\varphi$ is a formal power series in $q$.
We next show 
that the bigraded ring $\Rtri_{*,*}$ of $D_4$ triality invariants
is contained in a certain polynomial ring
(Lemma~\ref{lem:RtriDinv}). Then, the rest of the proof
proceeds in the same way as
in the case of  Weyl invariant $E_8$ Jacobi forms \cite{Sakai:2024vby}.
We make use of
the Seiberg--Witten curve for the $N_\mathrm{f}=4$ theory
\cite{Seiberg:1994aj}
\begin{align}
y^2=4x^3
 -\sum_{i=0}^2 a_i(\tau,\vecz)u^{2-i}x
 -\sum_{j=0}^3 b_j(\tau,\vecz)u^{3-j}.
\label{eq:SWcurveab0}
\end{align}
Here, $x,y\in\bbC$ are coordinates, $u\in\bbC$ is a parameter
(the moduli space coordinate)
and $a_i,b_j$ are (meromorphic) $D_4$ triality invariants
whose explicit forms are known.
The curve was originally introduced to describe
the moduli space of vacua of the gauge theory.
In this paper we merely use the properties of
functions $a_i,b_j$ to show that
the ring $\Rtri_{*,*}$ of $D_4$ triality invariants
is given by the intersection $\Rab\cap\Rcd$
of two polynomial rings
$\Rab=\bbC[a_i,b_j]$ and $\Rcd=\bbC[c_k,d_l]$
(Theorem~\ref{thm:RcaptR}).
Here $c_k$ and $d_l$ are functions related to
$a_i$ and $b_j$ in a simple manner.
We then construct an explicit isomorphism
between $\Rtri_{*,*}=\Rab\cap\Rcd$
and the ring of joint semiinvariants of
a binary quadratic and a binary cubic (Theorem~\ref{thm:RtriandS}).
Since the ring of semiinvariants
and the ring of covariants are isomorphic
(the Roberts isomorphism \cite{Roberts:1861}),
this completes the proof of Theorem~\ref{thm:main}.

A minimal basis of generators for the covariant ring
$\bbC[V_2\oplus V_3\oplus\bbC^2]^{\grp{SL}_2(\bbC)}$
has been known since the 19th century \cite{Gordan:1887}.
Applying our isomorphism to these generators,
we explicitly obtain the generators
for the ring $\Rtri_{*,*}$ of $D_4$ triality invariants.
Moreover, as an application of Theorem~\ref{thm:RcaptR}
we formulate an efficient algorithm for constructing
all $D_4$ triality invariants of given weight and degree.
Having such an algorithm is important not only practically
but also theoretically because the exhaustive search of invariants
is likely to bring us new insights
into the study of the $N_\mathrm{f}=4$ theory.

In our study of $D_4$ triality invariants
we make use of several ideas and techniques
originally developed for 
Jacobi forms (see e.g.~Lemmas~\ref{lem:freemodule},
\ref{lem:module}, \ref{lem:RtriDinv}).
This is possible because
$D_4$ triality invariants and Weyl invariant Jacobi forms
share important common features:
they are both invariants of a Weyl group and the modular group;
they both form a bigraded algebra over the ring of modular forms.
They differ in that
Jacobi forms are doubly quasi-periodic in $z_i$
while $D_4$ triality invariants are polynomial.
In this sense,
$D_4$ triality invariants may be viewed as polynomial analogs
of Jacobi forms.
Note that nontrivial polynomial analogs arise
due to the fiber product structure of $\Gtri$.

Among irreducible root systems, $D_4$ is the only one
that has the largest outer automorphism group $S_3$.
This means that
there are no other possibilities of $\RS$ triality invariants
as long as $\RS$ is an irreducible root system.
If we consider affine root systems,
$A_2^{(1)}$ and $E_6^{(1)}$
also admit an outer automorphism group $S_3$
and thus it is possible to define associated triality invariants.
In these cases, however,
invariant polynomials for an affine Weyl group
are only constants.
To construct nontrivial triality invariants,
one can use theta functions or characters instead,
but these give rise to non-polynomial invariants.
Thus, $D_4$ triality invariants are very unique polynomial invariants
in the place where reflection invariants and modular forms intersect.

There are several possible directions for further study.
Apart from the physical motivation,
we can relax the cusp condition and consider a wider class
of functions that are only required to be
a formal power series in $q^{1/2}$. We call these functions
weak $D_4$ triality invariants.
They naturally arise as coefficients
of the Taylor expansion of $D_4$ triality invariant Jacobi forms
\cite{Sakai:2023fnp}.
The weak $D_4$ triality invariants
also form a bigraded algebra over the ring of modular forms,
which is again not a polynomial algebra.
It is an interesting problem to determine the structure
of this algebra.
Other potentially related functions include
bimodular forms discussed recently
in the same context \cite{Aspman:2021evt}.

A more challenging but important problem
is the higher-rank generalization.
The correspondence we study in this paper arises
from $\mathcal{N}=2$ supersymmetric rank-one gauge theory,
which in general carries $\grp{SL}_2(\bbZ)$ (or its subgroup)
as the group of duality transformations.
For rank $r$ gauge theory the duality group
is promoted to (a subgroup of) $\grp{Sp}_{2r}(\bbZ)$.
It will be very interesting if we can establish
a similar correspondence that involves
Siegel modular forms and
invariants of higher-genus Seiberg--Witten curves.

At the beginning of this section we have seen three examples
$(n=1,2,4)$ of the correspondence between
the ring of automorphic forms and the invariant ring
$\bbC[V_{3n}\oplus V_{2n}]^{\grp{SL}_2(\bbC)}$
or the covariant ring
$\bbC[V_{3n}\oplus V_{2n}\oplus\bbC^2]^{\grp{SL}_2(\bbC)}$.
As automorphic forms, these examples are rather different 
(yet not totally unrelated \cite{Wang:2018fil,Sakai:2023fnp})
and the author does not know
whether there are more examples, in particular with other $n$.
What is common to these examples is that
each has a mirror symmetry behind it.
It will be intriguing if a unified perspective
becomes available for this kind of correspondences.

The structure of the paper is as follows.
In section~\ref{sec:Rtri}, we introduce $D_4$ triality invariants
and study basic structures of their modules and ring.
In section~\ref{sec:curve}, we recall the Seiberg--Witten curve
for the $N_\mathrm{f}=4$ theory
and show that the ring $\Rtri_{*,*}$ of $D_4$ triality invariants
is expressed as the intersection of two polynomial rings
generated by the curve coefficients.
In section~\ref{sec:iso}, we 
establish the explicit isomorphism between the ring $\Rtri_{*,*}$
and the ring of joint semiinvariants/covariants
of a binary cubic and a binary quadratic.
In section~\ref{sec:gen}, we study a minimal basis of
the covariant ring and that of $\Rtri_{*,*}$.
In Appendix~\ref{app:funct} we summarize
our convention of special functions.

\section{Basic properties of $D_4$ triality invariants}\label{sec:Rtri}

\subsection{Definitions}

Let $\bbH=\{\tau\in\bbC\,|\,\mathrm{Im}\,\tau>0\}$
be the upper half-plane and set $q=e^{2\pi i\tau}$.
We define the action of
$g=\left(\begin{smallmatrix}a&b\\ c&d\end{smallmatrix}\right)
\in\grp{SL}_2(\bbZ)$ on $\bbH$ by
\begin{align}
g\tau := \frac{a\tau+b}{c\tau+d}.
\end{align}
Let $M_k$ denote the space of modular forms of weight $k$ on
the full modular group $\grp{SL}_2(\bbZ)$.
In this paper the Eisenstein series $E_k(\tau)$ of weight $k$
is normalized as $E_k=1+O(q)$.
Modular forms on $\grp{SL}_2(\bbZ)$ form 
a graded ring $M_*=\bigoplus_{k=0}^\infty M_k$. We have
\begin{align}
M_*=\bbC[E_4,E_6].
\end{align}
The modular discriminant is written as
\begin{align}
\Delta=\eta^{24}=\frac{1}{1728}\left(E_4^3-E_6^2\right),
\label{eq:Delta}
\end{align}
where $\eta(\tau)$ is the Dedekind eta function.

Similarly, let $M_k(\Gamma(2))$ denote the space of
modular forms of weight $k$ on
the principal congruence subgroup $\Gamma(2)$.
The ring $M_*(\Gamma(2))=\bigoplus_{k=0}^\infty M_k(\Gamma(2))$
is freely generated by
$\varth_3^4=\varth_3(0,\tau)^4$ and $\varth_4^4=\varth_4(0,\tau)^4$
over $\bbC$:
\begin{align}
M_*(\Gamma(2))=\bbC[\varth_3^4,\varth_4^4].
\end{align}
Here $\varth_k(z,\tau)$ are Jacobi theta functions 
(see Appendix~\ref{app:funct}).

Let $W(\RS)$ denote the Weyl group associated with
an irreducible root system $\RS$ of rank $r$.
For every root $\vecal$ of $\RS$,
the reflection $s_\vecal$ is defined as
\begin{align}
s_\vecal:\ \bbC^r\to\bbC^r;\quad
\vecz\mapsto\vecz-2\frac{\vecal\cdot\vecz}{\vecal\cdot\vecal}\vecal,
\end{align}
where $\cdot$ denotes the Euclidean dot product.
The group $W(\RS)$ is generated by the reflections
associated with the roots of $\RS$.
In this paper we only consider two rank-4 root systems,
$\RS=D_4$ and $\RS=F_4$.
Let $\vece_i\ (i=1,2,3,4)$ be the standard basis of $\bbR^4$.
The roots of $D_4$ are given by
\begin{align}
\pm\vece_i\pm\vece_j\quad (1\le i<j\le 4).
\label{eq:D4roots}
\end{align}
These roots define reflections in $\bbC^4$. For instance, we have
\begin{align}
s_{\vece_1-\vece_2}:
\begin{pmatrix}
z_1\\ z_2\\ z_3\\ z_4
\end{pmatrix}
\ \mapsto\ 
\begin{pmatrix}
z_2\\ z_1\\ z_3\\ z_4
\end{pmatrix}
,\qquad
s_{\vece_1+\vece_2}:
\begin{pmatrix}
z_1\\ z_2\\ z_3\\ z_4
\end{pmatrix}
\ \mapsto\ 
\begin{pmatrix}
-z_2\\ -z_1\\ z_3\\ z_4
\end{pmatrix}
.
\end{align}
On the other hand, the roots of $F_4$ are given by
\begin{align}
\begin{aligned}
\mbox{long roots:}\quad&\pm\vece_i\pm\vece_j\quad (1\le i<j\le 4),\\
\mbox{short roots:}\quad&\pm\vece_i\quad (i=1,2,3,4),\qquad
\frac{1}{2}\left(\pm\vece_1\pm\vece_2\pm\vece_3\pm\vece_4\right).
\end{aligned}
\label{eq:F4roots}
\end{align}
As we can see, the long roots of $F_4$ are identical with
the roots of $D_4$.
Hence, obviously we have $W(D_4)\subset W(F_4)$.
Reflections with respect to short roots are written
as, for instance,
\begin{align}
\begin{aligned}
s_{\vece_4}:&
\begin{pmatrix}
z_1\\ z_2\\ z_3\\ z_4
\end{pmatrix}
\ \mapsto\ 
\begin{pmatrix}
z_1\\ z_2\\ z_3\\ -z_4
\end{pmatrix}
,\\
s_{\frac{1}{2}(\vece_1-\vece_2-\vece_3+\vece_4)}:&
\begin{pmatrix}
z_1\\ z_2\\ z_3\\ z_4
\end{pmatrix}
\ \mapsto\ 
\frac{1}{2}
\begin{pmatrix}
 z_1+z_2+z_3-z_4\\
 z_1+z_2-z_3+z_4\\
 z_1-z_2+z_3+z_4\\
-z_1+z_2+z_3+z_4
\end{pmatrix}
.
\end{aligned}
\end{align}
In fact, these two reflections generate
$W(F_4)/W(D_4)\cong S_3$.

As we saw in the last section,
let $\phi_1,\phi_2$ be group homomorphisms
given by the quotient maps
\begin{align}
\begin{aligned}
\phi_1&:\grp{SL}_2(\bbZ)\to\grp{SL}_2(\bbZ)/\Gamma(2)\cong S_3,\\
\phi_2&:W(F_4)\to W(F_4)/W(D_4)\cong S_3.
\end{aligned}
\end{align}
Identifying the two $S_3$'s we can define the fiber product
\begin{align}
\Gtri:=\grp{SL}_2(\bbZ)\times_{S_3}W(F_4)
 =\{(g,w)\in\grp{SL}_2(\bbZ)\times W(F_4)\,|\,\phi_1(g)=\phi_2(w)\}.
\label{eq:Gtri}
\end{align}
The multiplication of 
two elements $(g_1,w_1)$ and $(g_2,w_2)$ in $\Gtri$
is defined normally as
\begin{align}
(g_1,w_1)(g_2,w_2)=(g_1g_2,w_1w_2).
\end{align}
It is clear that $\Gtri$ satisfies
\begin{align}
\Gamma(2)\times W(D_4)
\ \subset\  \Gtri
\ \subset\
\grp{SL}_2(\bbZ)\times W(F_4).
\label{eq:subgrp}
\end{align}

To be specific, we fix the identification of two $S_3$'s
by the relations
\begin{align}
\phi_1\left(
\begin{pmatrix}
1&1\\
0&1
\end{pmatrix}\right)
=\phi_2(s_{\vece_4}),\qquad
\phi_1\left(
\begin{pmatrix}
0&-1\\
1&0
\end{pmatrix}\right)
=\phi_2\left(s_{\frac{1}{2}(\vece_1-\vece_2-\vece_3+\vece_4)}\right).
\end{align}
In other words, the action of the $S_3\in\Gtri$ on $\bbH\times\bbC^4$
is generated by
\begin{align}
\label{eq:Taction}
&\cT:&\tau&\mapsto\tau+1,
&&
\begin{pmatrix}
z_1\\ z_2\\ z_3\\ z_4
\end{pmatrix}
\ \mapsto\ 
\begin{pmatrix}
z_1\\ z_2\\ z_3\\ -z_4
\end{pmatrix},\\
\label{eq:Saction}
\hspace{4em}
&\cS:&\tau&\mapsto -\frac{1}{\tau},
&&
\begin{pmatrix}
z_1\\ z_2\\ z_3\\ z_4
\end{pmatrix}
\ \mapsto\ 
\frac{1}{2}
\begin{pmatrix}
 z_1+z_2+z_3-z_4\\
 z_1+z_2-z_3+z_4\\
 z_1-z_2+z_3+z_4\\
-z_1+z_2+z_3+z_4
\end{pmatrix}
.
\hspace{4em}
\end{align}
\begin{dfn}\label{def:triinv}
A $D_4$ triality invariant of weight $k$ and degree $m$
($k\in\bbZ,\ m\in \bbZ_{\ge 0}$) is a holomorphic function
$\varphi_{k,m}:\bbH\times\bbC^4\to \bbC$
that possesses the following properties:
\renewcommand{\theenumi}{\roman{enumi}}
\renewcommand{\labelenumi}{(\theenumi)}
\begin{enumerate}

\item $\varphi_{k,m}(\tau,z_1,z_2,z_3,z_4)$ is
a homogeneous polynomial of degree $m$ in $z_1,z_2,z_3,z_4$.

\item $\varphi_{k,m}$ transforms under the action of
$\left(\left(\begin{smallmatrix}a&b\\ c&d\end{smallmatrix}\right),
 w\right)\in\Gtri$ as
\begin{align}
&
\varphi_{k,m}\left(
\frac{a\tau+b}{c\tau+d}\,,w(\vecz)\right)
=(c\tau+d)^k\varphi_{k,m}(\tau,\vecz).
\end{align}

\item $\varphi_{k,m}$ satisfies the cusp condition
\begin{align}
\lim_{q\to 0}|\varphi_{k,m}(\tau,z_1,z_2,z_3,q^{-1/2}\Lambda)|<\infty,
\qquad q=e^{2\pi i\tau},
\label{eq:cuspcond}
\end{align}
where $\Lambda\in\bbC^*$ is a constant.
\end{enumerate}
Similarly, a weak $D_4$ triality invariant is
a function defined in the same way as above,
but with the following condition instead of (iii):
\renewcommand{\theenumi}{\roman{enumi}}
\renewcommand{\labelenumi}{(\theenumi)$'$}
\begin{enumerate}
\setcounter{enumi}{2}
\item $\varphi_{k,m}$ satisfies the cusp condition
\begin{align}
\lim_{q\to 0}|\varphi_{k,m}(\tau,\vecz)|<\infty.
\end{align}
\end{enumerate}
\end{dfn}
\begin{rem}
The cusp condition (iii) is equivalent to the requirement that
$\varphi_{k,m}$ has a Fourier expansion of the form
\begin{align}
\varphi_{k,m}(\tau,\vecz)
 =\sum_{n=0}^\infty v_n(\vecz)q^{n/2},\qquad
\mathrm{deg}_{z_4}v_n\le n.
\end{align}
Note that half-integral powers of $q$ can arise since
the change of sign $\cT(q^{l+1/2})=-q^{l+1/2}$
with $l\in\bbZ_{\ge 0}$
can be compensated by $\cT(v_{2l+1})=-v_{2l+1}$.
Note also that the condition is automatically satisfied for
$v_n$ with $n\ge m$, as all nonzero $v_n(\vecz)$ are
homogeneous polynomials of degree $m$ in $z_i$.
Similarly, the cusp condition (iii)$'$ is equivalent to the requirement
that $\varphi_{k,m}$ has a Fourier expansion of the form
\begin{align}
\varphi_{k,m}(\tau,\vecz)
 =\sum_{n=0}^\infty v_n(\vecz)q^{n/2}
\end{align}
without the restriction on $\mathrm{deg}_{z_4}v_n(\vecz)$.
Clearly, a $D_4$ triality invariant is
a weak $D_4$ triality invariant.
\end{rem}
\begin{rem}
Let us give some examples of $D_4$ triality invariants
and weak $D_4$ triality invariants.
\renewcommand{\theenumi}{\arabic{enumi}}
\renewcommand{\labelenumi}{$(\theenumi)$}
\begin{enumerate}
\item A modular form on $\grp{SL}_2(\bbZ)$ is
a $D_4$ triality invariant.

\item The quadratic polynomial $K=\sum_{i=1}^4 z_i^2$
is a weak $D_4$ triality invariant;
but it is not a $D_4$ triality invariant
since it does not satisfy the cusp condition (iii).
On the other hand,
$\Delta^n K\ (n\ge 1)$ are $D_4$ triality invariants.

\item We will later introduce
\begin{align}
L=\frac{\varth_3^4+\varth_4^4}{48}\sum_{1\le i<j\le 4}z_i^2z_j^2
 +\frac{-\varth_3^4+\varth_4^4}{8}\prod_{i=1}^4 z_i
 -\frac{\varth_3^4+\varth_4^4}{192}\left(\sum_{i=1}^4 z_i^2\right)^2.
\end{align}
This is a weak $D_4$ triality invariant,
but is not a $D_4$ triality invariant.
On the other hand, $\Delta^n L\ (n\ge 2)$ are $D_4$ triality invariants.

\end{enumerate}
\end{rem}
Let $\Rtri_{k,m}$ denote the $\bbC$-vector space of
$D_4$ triality invariants of weight $k$ and degree $m$.
We also introduce the notation
\begin{align}
\Rtri_{*,m}=\bigoplus_{k\in\bbZ}\Rtri_{k,m}.
\end{align}
It is clear from Definition~\ref{def:triinv}
that a $D_4$ triality invariant of degree $0$
is a modular form on $\grp{SL}_2(\bbZ)$, and vice versa.
In other words, we have
\begin{align}
\Rtri_{*,0}=M_*.
\end{align}
This allows us to view $\Rtri_{*,m}$ as a module over $M_*$.
It is also clear from Definition~\ref{def:triinv} that
$D_4$ triality invariants form a bigraded ring
\begin{align}
\Rtri_{*,*}=\bigoplus_{m=0}^\infty\Rtri_{*,m}
\end{align}
over $M_*$.

Similarly, we let $\Rwtri_{k,m}$ denote the $\bbC$-vector space of
weak $D_4$ triality invariants of weight $k$ and degree $m$.
We introduce the notations
\begin{align}
\Rwtri_{*,*}=\bigoplus_{m=0}^\infty\Rwtri_{*,m},\qquad
\Rwtri_{*,m}=\bigoplus_{k\in\bbZ}\Rwtri_{k,m},
\end{align}
which are viewed as
a bigraded ring and modules over $M_*$, respectively.
We have
\begin{align}
\Rtri_{k,m}\subset\Rwtri_{k,m},\qquad
\Rtri_{*,m}\subset\Rwtri_{*,m},\qquad
\Rtri_{*,*}\subset\Rwtri_{*,*}.
\end{align}
%

\subsection{Injection and cusp condition}

In this subsection we introduce the injection $\Sop$
which simplifies checking the cusp condition
of $D_4$ triality invariants.

Let $\RD_m$ be the $\bbC$-vector space of
invariant polynomials of $W(D_4)$ of degree $m$.
By the Chevalley--Shephard--Todd theorem, the graded ring
\begin{align}
\RD_*=\bigoplus_{m=0}^\infty\RD_m
\end{align}
is a polynomial ring generated by
four invariant polynomials of degree $2,4,6,4$
\begin{align}
\RD_*=\bbC[I_2,I_4,I_6,\tI_4].
\end{align}
We take the basis elements as
\begin{align}
I_2&=\sum_{i=1}^4 z_i^2,\qquad
I_4&=\sum_{1\le i<j\le 4}z_i^2z_j^2,\qquad
I_6&=\sum_{1\le i<j<k\le 4}z_i^2z_j^2z_k^2,\qquad
\tI_4&=\prod_{i=1}^4 z_i.
\end{align}
These are algebraically independent over $\bbC$, as clearly
seen from the Jacobian determinant:
\begin{align}
\left|\frac{\partial(I_2,I_4,I_6,\tI_4)}{\partial(z_1,z_2,z_3,z_4)}
\right|
 =8\prod_{1\le i<j\le 4}(z_i^2-z_j^2).
\label{eq:JacobiI}
\end{align}

The following proposition is also clear from 
the relation \eqref{eq:subgrp}:
\begin{prop}\label{prop:G2WD4ring}
We have
\begin{align}
\Rtri_{*,*}\subset
\Rwtri_{*,*}\subset
M_*(\Gamma(2))\otimes_\bbC\RD_*
=\bbC[\varth_3^4,\varth_4^4,I_2,I_4,I_6,\tI_4].
\end{align}
\end{prop}
From this we immediately see that
the weight and index of every nonzero (weak) $D_4$ triality invariant
are both non-negative even numbers. In other words,
\begin{align}
\Rtri_{k,m}=\Rwtri_{k,m}=\{0\}\quad\mbox{unless}\quad
k,m\in 2\bbZ_{\ge 0}.
\end{align}
Before introducing the injection $\Sop$,
let us prove the following lemma:
\begin{lem}\label{lem:qexp}
Every weak $D_4$ triality invariant $\varphi\in\Rwtri_{*,*}$
can be expanded as
\begin{align}
\varphi=
 \sum_{n=0}^\infty v_{1,n}(\vecz)q^n
+\tI_4\sum_{n=0}^\infty v_{2,n}(\vecz)q^{n+1/2},
\end{align}
where
\begin{align}
v_{a,n}\in\bbC[I_2,I_4,I_6,\tI_4^2].
\end{align}
\end{lem}
\begin{proof}
By Proposition~\ref{prop:G2WD4ring}
and the fact that $\varth_3^4,\varth_4^4$ have
a Fourier expansion of the form $\sum_{n=0}^\infty\alpha_nq^{n/2}$,
the Fourier expansion of $\varphi$ can be expressed as
\begin{align}
\varphi= \sum_{n=0}^\infty v_n(\vecz)q^n
+\sum_{n=0}^\infty \tilde{v}_n(\vecz)q^{n+1/2},\qquad
v_n,\tilde{v}_n\in\bbC[I_2,I_4,I_6,\tI_4].
\end{align}
Since $\varphi$ is a weak $D_4$ triality invariant,
$\varphi$ is invariant under the action of $\cT$
given in \eqref{eq:Taction}.
This implies that $\cT(v_n)=v_n$,
$\cT(\tilde{v}_n)=-\tilde{v}_n$
and thus $v_n$ ($\tilde{v}_n$) are
even (odd) polynomials of $\tI_4$.
\end{proof}

Let us now define the mapping $\Sop$
which acts on weak $D_4$ triality invariants by
\begin{align}
\Sop(q,I_2,I_4,I_6,\tI_4)
=\left(q,
 \frac{I_2}{q},\frac{I_4}{q},\frac{I_6}{q},\frac{\tI_4}{q^{1/2}}\right).
\end{align}
Clearly, $\Sop$ is injective and gives an automorphism
of the ring of formal Laurent series of $q^{1/2}$ over $\RD_*$.
This injection $\Sop$ plays a crucial role in this paper,
as it gives us a transparent way to check
the cusp condition of $D_4$ triality invariants:
\begin{lem}\label{lem:mqexp}
Let $\varphi$ be a weak $D_4$ triality invariant.
Then the following two conditions are equivalent:
\renewcommand{\theenumi}{\roman{enumi}}
\renewcommand{\labelenumi}{$\mathrm{(\theenumi})$}
\begin{enumerate}
\item $\varphi$ is a $D_4$ triality invariant.

\item $\Sop\varphi$ is a formal power series in $q$,
i.e.~it is expanded as
\begin{align}
\Sop\varphi
 =\sum_{n=0}^\infty v_n(\vecz)q^n,\qquad v_n\in\RD_*.
\label{eq:subsqexp}
\end{align}
\end{enumerate}
\end{lem}
\begin{proof}
(i) $\Rightarrow$ (ii):
By Lemma~\ref{lem:qexp} we immediately see that
for any $\varphi\in\Rwtri_{*,*}$,
the image
$\Sop\varphi$ is written as a formal Laurent series
$\Sop\varphi=\sum_{n=n_0}^\infty \tilde{v}_n(\vecz)q^n$
with $\tilde{v}_n\in\RD_*$.
It therefore suffices to show that $n_0\ge 0$
if $\varphi$ is in $\Rtri_{*,*}$.
By Proposition~\ref{prop:G2WD4ring},
$\varphi$ can be written as a finite sum
\begin{align}
\varphi=\sum_i f_i(\tau)u_i(\vecz),
\end{align}
where $f_i\in\bbC[\varth_3^4,\varth_4^4]$
and $u_i$ are monomials of $I_2,I_4,I_6,\tI_4$.
Then,
for each $u_i$
we have $\Sop u_i=q^{-\delta_i/2}u_i$,
where $\delta_i:=\mathrm{deg}_{z_4}(u_i)$.
(This can be seen as follows: if we write $u_i$ as
$u_i=I_2^aI_4^bI_6^c\tI_4^d\ (a,b,c,d\in\bbZ_{\ge 0})$,
we have
\begin{align}
\begin{aligned}
\delta_i
 =\mathrm{deg}_{z_4}u_i
&=a\,\mathrm{deg}_{z_4}I_2
 +b\,\mathrm{deg}_{z_4}I_4
 +c\,\mathrm{deg}_{z_4}I_6
 +d\,\mathrm{deg}_{z_4}\tI_4\\
&=2a+2b+2c+d,
\end{aligned}
\end{align}
while $\Sop u_i = q^{-a-b-c-d/2}u_i$.)
On the other hand, if $\varphi$ is in $\Rtri_{*,*}$,
then each $f_i$ 
must have a Fourier expansion
$f_i=O(q^{\delta_i/2})$,
i.e.~it has no terms of the order of $q^\alpha$
with $\alpha<\delta_i/2$,
due to the property (iii) of Definition~\ref{def:triinv}.
Thus, $\Sop(\varphi)=\sum_i\Sop(f_i u_i)=O(q^0)$.

(ii) $\Rightarrow$ (i):
It suffices to show that $\varphi$ has the property (iii) of
Definition~\ref{def:triinv}.
Since $\Sop$ is injective, we have
\begin{align}
\varphi
 =\sum_{n=0}^\infty \Sop^{-1}(v_n)q^n.
\end{align}
Recall that $v_n$ are polynomials of $I_2,I_4,I_6,\tI_4$
and for every $X=I_2,I_4,I_6,\tI_4$ it is easy to see that
$|\Sop^{-1}(X)|$ with $z_4=q^{-1/2}\Lambda$
stays finite in the limit of $q\to 0$.
This implies that $\varphi$ possesses
the property (iii) of Definition~\ref{def:triinv}.
\end{proof}
%

\subsection{Free modules}

In this subsection we study the structure of
the module $\Rtri_{*,m}$ of $D_4$ triality invariants.

\begin{lem}\label{lem:freemodule}
$\Rtri_{*,m}$ is a free module over $M_*$.
\end{lem}
\begin{proof}
We can prove the lemma in the same way as
in the case of Jacobi forms \cite[Theorem 8.4]{EichlerZagier}.
We do it by induction.
Assume that for some $k\in 2\bbZ_{>0}$
we have found $D_4$ triality invariants
$\varphi_i\ (i=1,\ldots,r)$ of weight $k_i<k$
such that they form a free basis
for $\Rtri_{*,m}$ over $M_*$ in weights less than $k$.
This means that every $\varphi\in\Rtri_{k',m}$ for $k'<k$
can be written uniquely as $\varphi=\sum_{i=1}^r f_i\varphi_i$
with $f_i\in M_{k-k'}$.
This induction hypothesis
is trivially satisfied if $k$ is the smallest integer
with $\Rtri_{k,m}\ne\{0\}$.
If we can show that $\{\varphi_i\}_{i=1}^r$ is also linearly
independent in $\Rtri_{k,m}$ over $M_*$, then we are done:
by adding $\varphi_{i+1},\ldots,\varphi_s$
which form a $\bbC$-basis for the complement of
$\sum_{i=1}^r\varphi_iM_{k-k_i}$ in $\Rtri_{k,m}$,
we obtain a new free basis $\varphi_i\ (i=1,\ldots,s)$
for the induction hypothesis with $k$ replaced by $k+2$.

Therefore, suppose that we have a relation
\begin{align}\label{eq:fphirel}
\sum_{i=1}^rf_i\varphi_i=0
\end{align}
in $\Rtri_{k,m}$ with $f_i\in M_{k-k_i}$.
Since $k-k_i>0$, each $f_i$ can be written as
\begin{align}
f_i=E_4g_i+E_6h_i
\end{align}
for some modular forms $g_i\in M_{k-4-k_i}$
and $h_i\in M_{k-6-k_i}$.
(Note that $g_i,h_i$ could be zero.)
The relation \eqref{eq:fphirel} is then written as
\begin{align}
E_4\sum_{i}g_i\varphi_i+E_6\sum_{i}h_i\varphi_i=0.
\end{align}
But this implies that
there exists a $D_4$ triality invariant
$\psi\in\Rtri_{k-10,m}$ such that
\begin{align}
\psi
 =\frac{\sum_{i}g_i\varphi_i}{E_6}
 =-\frac{\sum_{i}h_i\varphi_i}{E_4}.
\end{align}
Here, $\psi$ lies in $\Rtri_{k-10,m}$ because
it is holomorphic at the zero points of $E_4$ ($E_6$)
by the first (second) equality
and $\Sop\psi$ is a formal power series in $q$.
Since $\psi$ is of weight $k-10$, which is less than $k$,
by our induction hypothesis 
$\psi$ is uniquely written as
\begin{align}
\psi=\sum_{i=1}^r\tf_i\varphi_i
\end{align}
with $\tf_i\in M_{k-10-k_i}$. Then we have
\begin{align}
\sum_{i}(g_i-E_6\tf_i)\varphi_i=0,\qquad
\sum_{i}(h_i+E_4\tf_i)\varphi_i=0.
\end{align}
Since $\varphi_i$ are linearly
independent in $\Rtri_{k-4,m}$ and in $\Rtri_{k-6,m}$,
the above relations give
$g_i=E_6\tf_i,$ $h_i=-E_4\tf_i$,
which lead to $f_i=0\ (i=1,\ldots,r)$.
\end{proof}

To take a closer look on the module $\Rtri_{*,m}$,
it is convenient to 
describe it in terms of specific weak $D_4$ triality invariants.
To do this, let us first introduce fundamental building blocks
$e_i\in M_2(\Gamma(2)),\ T_i\in\RD_4\ (i=1,2,3)$
as follows:
\begin{align}
e_1=
 \frac{1}{12}\left(\varth_3^4+\varth_4^4\right),\qquad
e_2=
 \frac{1}{12}\left(\varth_2^4-\varth_4^4\right),\qquad
e_3=
 \frac{1}{12}\left(-\varth_2^4-\varth_3^4\right),
\end{align}
\begin{align}
T_1=\frac{1}{6}I_4-\frac{1}{24}I_2^2,\qquad
T_2=-\frac{1}{12}I_4-\frac{1}{2}\tI_4+\frac{1}{48}I_2^2,\qquad
T_3=-\frac{1}{12}I_4+\frac{1}{2}\tI_4+\frac{1}{48}I_2^2.
\end{align}
They are subject to the constraints
\begin{align}
e_1+e_2+e_3=0,\qquad T_1+T_2+T_3=0
\end{align}
and satisfy
\begin{align}
\begin{aligned}
\cT(e_1,e_2,e_3)&=(e_1,e_3,e_2),
 \quad&\cT(T_1,T_2,T_3)&=(T_1,T_3,T_2),\\
\cS(e_1,e_2,e_3)&=\tau^2(e_3,e_2,e_1),
 \quad&\cS(T_1,T_2,T_3)&=(T_3,T_2,T_1).
\end{aligned}
\end{align}
From this we immediately obtain
a family of weak $D_4$ triality invariants
\begin{align}
\sum_{i=1}^3 e_i^l T_i^n \in\Rwtri_{2l,4n}
\qquad(l,n\in\bbZ_{\ge 0}).
\label{eq:eTgen}
\end{align}

We now
introduce the four fundamental weak $D_4$ triality invariants
\begin{align}
\begin{aligned}
K&=I_2,\\
L&= e_1 T_1 + e_2 T_2 + e_3 T_3,\\
M&= 12(e_1^2 T_1 + e_2^2 T_2 + e_3^2 T_3),\\
N&=\frac{1}{4}I_6-\frac{1}{24}I_2I_4+\frac{1}{96}I_2^3.
\end{aligned}
\end{align}
We choose them in such a way that
they have small weight and degree
and they are algebraically independent over $M_*$.
More precisely, the following proposition holds:
\begin{prop}\label{prop:KLMNindep}
$K,L,M,N$ are algebraically independent
for each point $\tau_0\in\bbH$.
\end{prop}
\begin{proof}
By using \eqref{eq:JacobiI},
the Jacobian of $K,L,M,N$ is easily calculated as
\begin{align}
\left|\frac{\partial(K,L,M,N)}{\partial(z_1,z_2,z_3,z_4)}\right|
 =-\frac{\eta^{12}}{2}\prod_{1\le i<j\le 4}(z_i^2-z_j^2).
\end{align}
Since $\eta(\tau)$ vanishes nowhere in $\bbH$,
this in particular means that
$K,L,M,N$ are algebraically independent
at every point $\tau=\tau_0$ in $\bbH$.
\end{proof}

We can now prove the following dimension formula:
\begin{lem}\label{lem:module}
The rank $r(m)$ of $\Rtri_{*,m}$ over $M_*$ is given by
\begin{align}
\sum_{m=0}^\infty r(m)x^m=\frac{1}{(1-x^2)(1-x^4)^2(1-x^6)}.
\end{align}
\end{lem}
\begin{proof}
Our proof is similar to that of
\cite[Theorem 4.1 (Lemma 4.3)]{Wang:2018fil}.
Let $\rho(m)$ be the rank of $\Rtri_{*,m}$
and $r(m)$ be given as above.
Being multiplied by a sufficiently large power of $\Delta$,
any polynomial of $K,L,M,N$
gives a $D_4$ triality invariant.
Since $K,L,M,N$ are of degree $2,4,4,6$, respectively
and they are algebraically independent over $M_*$,
we have $\rho(m)\ge r(m)$.

Now, suppose that $\rho(m)>r(m)$.
Let $\{\varphi_i\}\ (i=1,\ldots,\rho(m))$ be a free basis
of $\Rtri_{*,m}$.
Let $k_i$ be the weight of $\varphi_i$ and 
$k_\mathrm{max}$ the largest weight among
$k_1,\ldots,k_{\rho(m)}$.
By Lemma~\ref{lem:mqexp},
$\Sop\varphi_i$ has a regular $q$-expansion and 
the constant term of $\Sop\varphi_i$
(with respect to $q$)
lies in $\RD_m$.
Since the ring $\RD_*$ is freely generated by
$I_2,I_4,I_6,\tI_4$
of degree $2,4,6,4$ respectively,
the rank of $\RD_m$ is given by $r(m)$.
By our assumption $\rho(m)>r(m)$, this means that
the constant terms of $\Sop\varphi_i\ (i=1,\ldots,\rho(m))$
are linearly dependent.
Thus, there exists a nonzero linear combination
\begin{align}
\varphi=\sum_{i=1}^{\rho(m)}
 c_iE_{4+k_\mathrm{max}-k_i}\varphi_i,\quad c_i\in\bbC
\end{align}
in $\Rtri_{4+k_\mathrm{max},m}$
such that the constant term of $\Sop\varphi$ vanishes.
Then $\Sop\varphi/\Delta$ has a regular $q$-expansion
and hence $\varphi/\Delta$ is a $D_4$ triality invariant
by Lemma~\ref{lem:mqexp}.
Therefore $\varphi/\Delta$ can be expressed as
a linear combination
\begin{align}
\frac{\varphi}{\Delta}=\sum_{i=1}^{\rho(m)}f_i\varphi_i,\quad
f_i\in M_{-8+k_\mathrm{max}-k_i}.
\end{align}
However, this means that
$c_iE_{4+k_\mathrm{max}-k_i}=\Delta f_i$
and thus $c_i=0$ for all $i$,
leading to a contradiction.
Hence we have $\rho(m)=r(m)$.
\end{proof}

We end this section with the following lemma,
which plays an important role in the next section:
\begin{lem}\label{lem:RtriDinv}
The ring of $D_4$ triality invariants
is contained in the polynomial ring generated by
$K,L,M,N$ over $M_*[\Delta^{-1}]$:
\begin{align}
\Rtri_{*,*}\subset
 M_*\left[\frac{1}{\Delta}\right][K,L,M,N].
\end{align}
\end{lem}
\begin{proof}
The lemma is proved in a similar manner as
\cite[Theorem 8.3]{EichlerZagier}.
Let $\varphi$ be a $D_4$ triality invariant of degree $m$
and $Q_i\ (i=1,\ldots,r(m))$ be independent monomials
of $K,L,M,N$ of degree $m$ (in $z_i$).
By Lemma~\ref{lem:module},
$\varphi$ and $Q_1,\ldots,Q_{r(m)}$ cannot be linearly independent
over $M_*$ and hence, they satisfy a linear relation
\begin{align}
f(\tau)\varphi=\sum_{i=1}^{r(m)}f_i(\tau)Q_i(K,L,M,N)
\label{eq:gphirep}
\end{align}
with $f\ne 0$, where $f,f_i\in M_*$.
Suppose that $f$ is of minimal weight
such that a relation of the form \eqref{eq:gphirep} holds.
If $f$ is not a power of $\Delta$,
then $f$ must vanish at some point $\tau_0\in\bbH$.
Since $Q_i$ are linearly independent
(for any $\tau_0\in\bbH$ by Proposition~\ref{prop:KLMNindep}),
this means that $f_i(\tau_0)=0$ for all $i=1,\ldots,r(m)$.
But then $f$ and $f_i$ have a common factor
(which is 
$E_4$ or $E_6$ if $\tau_0$ is
a zero point of $E_4$ or $E_6$,
or $(E_4^3-j(\tau_0)\Delta)$ otherwise,
where $j(\tau):=E_4(\tau)^3/\Delta(\tau)$).
This contradicts the minimality of $f$.
\end{proof}
%

\section{Seiberg--Witten curve and intersection of polynomial rings}
\label{sec:curve}

In this section we recall the Seiberg--Witten curve for
the $N_\mathrm{f}=4$ theory and discuss properties of its coefficients.
We then show that the ring $\Rtri_{*,*}$ of
$D_4$ triality invariants is given by the intersection
of two polynomial rings generated by these coefficients.

\subsection{Seiberg--Witten curve and its coefficients}

The Seiberg--Witten curve for
the $\mathcal{N}=2$ supersymmetric $\grp{SU}(2)$ gauge theory
with four flavors (the $N_\mathrm{f}=4$ theory)
is given by\footnote{While the curve \eqref{eq:SWoriginal}
is apparently identical with
the original one \cite[eqs.(16.35)--(16.36)]{Seiberg:1994aj},
our convention of $e_i$'s is slightly different from theirs.
The two curves can be identified completely
if we perform an additional transformation of variables:
we identify $\tau^\mathrm{there} = \tau^\mathrm{here}+1$,
so that we have
$e^\mathrm{there}_i(\tau^\mathrm{there})
 =4e^\mathrm{here}_i(\tau^\mathrm{here})\ (i=1,2,3)$.
The difference in the normalization of $e_i$ is compensated by
rescaling the variables as
$\widetilde{u}^\mathrm{there} = 4u^\mathrm{here},$
$x^\mathrm{there} = 16x^\mathrm{here},$
$y^\mathrm{there} = 32y^\mathrm{here}.$
}
\begin{align}
y^2=\ &
 4[W_1W_2W_3+A(W_1T_1(e_2-e_3)+W_2T_2(e_3-e_1)+W_3T_3(e_1-e_2))-A^2N],
\nn\\
&W_k=x-e_ku-\frac{e_k^2}{2}K,\qquad
 A=(e_1-e_2)(e_2-e_3)(e_3-e_1).
\label{eq:SWoriginal}
\end{align}
The curve was determined in \cite{Seiberg:1994aj}
so that it describes
the moduli space of vacua of the gauge theory.

We perform a translation of $x$ in such a way that
the quadratic terms in $x$ vanish,
then make a translation of $u$
so that the terms of the order of $ux$ vanish.
The curve then takes the form
\begin{align}
y^2=4x^3-(a_0u^2+a_2)x-(b_0u^3+b_1u^2+b_2u+b_3),
\label{eq:abcurve}
\end{align}
where
\begin{align}
\begin{aligned}
a_0&=\frac{E_4}{12},&
a_2&=\frac{\Delta K^2}{4E_4}-\frac{E_6L}{24}+\frac{E_4M}{24},
 \hspace{-10em}&&\\
b_0&=\frac{E_6}{216},\quad&
b_1&=\frac{\Delta K}{E_4},\quad&
b_2&=
 -\frac{E_6\Delta K^2}{24E_4^2}-\frac{E_4^2L}{288}+\frac{E_6M}{288},\\
b_3&=
 -\frac{\Delta^2K^3}{E_4^3}+\frac{\Delta KM}{4E_4}+\frac{\Delta N}{4}.
\hspace{-10em}
\end{aligned}
\label{eq:abform}
\end{align}
Observe that $a_2,b_1,b_2,b_3$ satisfy all the properties
of $D_4$ triality invariants except that they
are meromorphic in $\tau$.
We will summarize their properties later.

Let us introduce another natural form of the curve:
We again perform a translation of $u$ so that
the terms of the order of $u^2$ vanish. Then we obtain
\begin{align}
y^2=4x^3-(c_0u^2+c_1u+c_2)x-(d_0u^3+d_2u+d_3),
\label{eq:cdcurve}
\end{align}
where
\begin{align}
\begin{aligned}
c_0&=\frac{E_4}{12},&
c_1&=-\frac{12\Delta K}{E_6},&
c_2&=\frac{E_4^2\Delta K^2}{4E_6^2}-\frac{E_6L}{24}+\frac{E_4M}{24},\\
d_0&=\frac{E_6}{216},\quad&
d_2&=-\frac{E_4\Delta K^2}{24E_6}-\frac{E_4^2L}{288}+\frac{E_6M}{288},
 \quad&
d_3&=\frac{2\Delta^2K^3}{E_6^2}+\frac{E_4\Delta KL}{4E_6}
 +\frac{\Delta N}{4}.
\end{aligned}
\label{eq:cdform}
\end{align}
Note that $c_1,c_2,d_2,d_3$ are also
meromorphic $D_4$ triality invariants.
It is easy to see that $a_i,b_j$ and $c_k,d_l$
are related as
\begin{align}
a_i&=\sum_{j=0}^i c_j
 \begin{pmatrix} 2-j\\[1ex] 2-i\end{pmatrix}
 \left(-\frac{c_1}{2c_0}\right)^{i-j},\quad&&
b_i=\sum_{j=0}^i d_j
 \begin{pmatrix} 3-j\\[1ex] 3-i\end{pmatrix}
 \left(-\frac{c_1}{2c_0}\right)^{i-j}.
\label{eq:abincd}
\end{align}
The inverse mapping is written as
\begin{align}
c_i&=\sum_{j=0}^i a_j
 \begin{pmatrix} 2-j\\[1ex] 2-i\end{pmatrix}
 \left(-\frac{b_1}{3b_0}\right)^{i-j},&&
d_i=\sum_{j=0}^i b_j
 \begin{pmatrix} 3-j\\[1ex] 3-i\end{pmatrix}
 \left(-\frac{b_1}{3b_0}\right)^{i-j}.
\end{align}

We now summarize the properties of $a_i,b_j,c_k,d_l$.
\begin{prop}
\label{prop:abprop}
The properties of $a_i$ and $b_j$ are summarized as follows.
\renewcommand{\theenumi}{\arabic{enumi}}
\renewcommand{\labelenumi}{$(\theenumi)$}
\begin{enumerate}
\item Each of $a_i,b_j$ is homogeneous with respect to the bigrading:
$a_i$ is of weight $4+2i$ and degree $2i$;
$b_j$ is of weight $6+2j$ and degree $2j$.

\item The functions $a_2,b_1,b_2,b_3$
possess all the properties of
$D_4$ triality invariants, except that
they have poles at the zero points of $E_4$.
When multiplied by a certain power of $E_4$,
they give $D_4$ triality invariants:
\begin{align}
E_4a_2\in\Rtri_{12,4},\qquad
E_4^j b_j\in\Rtri_{6+6j,2j}.
\end{align}

\item $\Sop a_i,\Sop b_j$ admit a power series expansion
of the form
\begin{align}
\Sop a_i(\tau,\vecz)=\sum_{n=0}^\infty a_i^{(n)}(\vecz)q^n,\qquad
\Sop b_j(\tau,\vecz)=\sum_{n=0}^\infty b_j^{(n)}(\vecz)q^n.
\label{eq:abqexp}
\end{align}
\end{enumerate}
\end{prop}
\begin{prop}
\label{prop:cdprop}
The properties of $c_i$ and $d_j$ are summarized as follows.
\renewcommand{\theenumi}{\arabic{enumi}}
\renewcommand{\labelenumi}{$(\theenumi)$}
\begin{enumerate}
\item Each of $c_i,d_j$ is homogeneous with respect to the bigrading:
$c_i$ is of weight $4+2i$ and degree $2i$;
$d_j$ is of weight $6+2j$ and degree $2j$.

\item The functions $c_1,c_2,d_2,d_3$
possess all the properties of
$D_4$ triality invariants, except that
they have poles at the zero points of $E_6$.
When multiplied by a certain power of $E_6$,
they give $D_4$ triality invariants:
\begin{align}
E_6^ic_i\in\Rtri_{4+8i,2i},\qquad
E_6^{j-1} d_j\in\Rtri_{8j,2j}.
\end{align}

\item $\Sop c_i,\Sop d_j$ admit a power series expansion
of the form
\begin{align}
\Sop c_i(\tau,\vecz)=\sum_{n=0}^\infty c_i^{(n)}(\vecz)q^n,\qquad
\Sop d_j(\tau,\vecz)=\sum_{n=0}^\infty d_j^{(n)}(\vecz)q^n.
\label{eq:cdqexp}
\end{align}
\end{enumerate}
\end{prop}

The leading-order coefficients of the expansions
\eqref{eq:abqexp} and \eqref{eq:cdqexp}
are calculated as
\begin{align}
\begin{aligned}
a_0^{(0)}&= \frac{1}{12},&
a_2^{(0)}&= I_4+\frac{\tI_4}{4}-64I_2^2,\\
b_0^{(0)}&= \frac{1}{216},&
b_1^{(0)}&= I_2,\\
b_2^{(0)}&= -\frac{I_4}{6}+\frac{\tI_4}{48}+\frac{128I_2^2}{3},\quad&
b_3^{(0)}&= \frac{I_6}{16}-4I_2I_4+I_2\tI_4+512I_2^3,
\end{aligned}
\label{eq:abq0form}
\end{align}
\begin{align}
\begin{aligned}
c_0^{(0)}&= \frac{1}{12},&
c_1^{(0)}&= -12I_2,&
c_2^{(0)}&= I_4+\frac{\tI_4}{4}+368I_2^2,\\
d_0^{(0)}&= \frac{1}{216},\quad&
d_2^{(0)}&= -\frac{I_4}{6}+\frac{\tI_4}{48}-\frac{88I_2^2}{3},\quad&
d_3^{(0)}&= \frac{I_6}{16}+8I_2I_4-\frac{I_2\tI_4}{2}+896I_2^3.
\end{aligned}
\label{eq:cdq0form}
\end{align}
The following propositions hold:
\begin{prop}\label{prop:abcdq0indep}
The coefficients
$a_2^{(0)},b_1^{(0)},b_2^{(0)},b_3^{(0)}$
are algebraically independent over $\bbC$.
Similarly,
$c_1^{(0)},c_2^{(0)},d_2^{(0)},d_3^{(0)}$
are algebraically independent over $\bbC$.
\end{prop}
\begin{proof}
Using \eqref{eq:JacobiI} we can calculate the Jacobians as
\begin{align}
\begin{aligned}
\left|\frac{\partial(a_2^{(0)},b_1^{(0)},b_2^{(0)},b_3^{(0)})}
           {\partial(z_1,z_2,z_3,z_4)}\right|
&
 =\frac{1}{32}\prod_{1\le i<j\le 4}(z_i^2-z_j^2),\\
\left|\frac{\partial(c_1^{(0)},c_2^{(0)},d_2^{(0)},d_3^{(0)})}
           {\partial(z_1,z_2,z_3,z_4)}\right|
&
 =\frac{3}{8}\prod_{1\le i<j\le 4}(z_i^2-z_j^2),
\end{aligned}
\end{align}
which are not identically zero.
\end{proof}
\begin{prop}\label{prop:abcdindep}
The functions $a_0,a_2,b_0,b_1,b_2,b_3$ are algebraically independent
over $\bbC$. Similarly, $c_0,c_1,c_2,d_0,d_2,d_3$ are algebraically
independent over $\bbC$.
\end{prop}
\begin{proof}
Using the expressions \eqref{eq:abform} and \eqref{eq:cdform}
we can calculate the Jacobian determinants
\begin{align}
\begin{aligned}
\left|\frac{\partial(a_2,b_1,b_2,b_3)}{\partial(K,L,M,N)}\right|
 =-\frac{\Delta^3}{16E_4},\qquad
\left|\frac{\partial(c_1,c_2,d_2,d_3)}{\partial(K,L,M,N)}\right|
 =-\frac{3\Delta^3}{4E_6}.
\end{aligned}
\end{align}
By the proof of Proposition~\ref{prop:KLMNindep}
we see that $K,L,M,N$ are algebraically independent
over $M_*=\bbC[a_0,b_0]=\bbC[c_0,d_0]$.
\end{proof}
%

\subsection{Intersection of polynomial rings}

The goal of this subsection is to 
give a sharp description of the ring $\Rtri_{*,*}$
of $D_4$ triality invariants
as the intersection of two polynomial rings.

We first note that the expressions \eqref{eq:abform}
and \eqref{eq:cdform} can be inverted, so that we can write
$K,L,M,N$ in terms of $a_i,b_j$ or in terms of $c_k,d_l$.
When multiplied by a sufficiently large power of
$\Delta(=a_0^3-27b_0^2=c_0^3-27d_0^2)$,
they are written as polynomials:
\begin{align}
\begin{aligned}
\Delta K
 &=12a_0b_1,\\
\Delta^2L
 &=
 -2a_0^4b_2+3a_0^3a_2b_0+54a_0b_0^2b_2-27a_0b_0b_1^2-81a_2b_0^3,\\
 \Delta^2M
 &=
 2a_0^5a_2-36a_0^3b_0b_2-6a_0^3b_1^2-54a_0^2a_2b_0^2+972b_0^3b_2
 -324b_0^2b_1^2,\\
\Delta^3N
 &=
 4a_0^6b_3-2a_0^5a_2b_1-216a_0^3b_0^2b_3+36a_0^3b_0b_1b_2+10a_0^3b_1^3
 +54a_0^2a_2b_0^2b_1\\
&\hspace{1em}
 +2916b_0^4b_3-972b_0^3b_1b_2+216b_0^2b_1^3,
\end{aligned}
\label{eq:KLMNinab}
\end{align}
\begin{align}
\begin{aligned}
\Delta K
 &=-18c_1d_0,\\
\Delta^2L
 &=
 -2c_0^4d_2+3c_0^3c_2d_0-\frac{9}{4}c_0^2c_1^2d_0+54c_0d_0^2d_2
 -81c_2d_0^3,\\
\Delta^2M
 &=
 2c_0^5c_2-\frac{1}{2}c_0^4c_1^2-36c_0^3d_0d_2-54c_0^2c_2d_0^2
 -27c_0c_1^2d_0^2+972d_0^3d_2,\\
\Delta^3N
 &=
 4c_0^6d_3-2c_0^5c_1d_2+3c_0^4c_1c_2d_0-\frac{5}{4}c_0^3c_1^3d_0
 -216c_0^3d_0^2d_3+54c_0^2c_1d_0^2d_2\\
 &\hspace{1em}
 -81c_0c_1c_2d_0^3-27c_1^3d_0^3+2916d_0^4d_3.
\end{aligned}
\label{eq:KLMNincd}
\end{align}

Let $\Rab$ be the bigraded polynomial ring
(graded by weight and degree in $z_i$)
\begin{align}
\Rab:=\bbC[a_0,a_2,b_0,b_1,b_2,b_3].
\end{align}
By Proposition~\ref{prop:abcdindep}, $\Rab$ is freely generated.
The following lemma is a key to the proof of our main theorem
in this subsection:
\begin{lem}\label{lem:RGinRab}
The ring $\Rtri_{*,*}$ of $D_4$ triality invariants
is contained in $\Rab$:
\begin{align}
\Rtri_{*,*}\subset \Rab.
\end{align}
In other words, every $D_4$ triality invariant
is expressed uniquely as a polynomial of
$a_0,a_2,b_0,b_1,b_2,b_3$ over $\bbC$. 
\end{lem}
\begin{proof}
The lemma can be proved in a similar manner
as \cite[Theorem 3.1]{Sakai:2022taq},
which relies on \cite[Theorem 1.1]{Sun:2021ije}.\footnote{Theorem~3.1
of \cite{Sakai:2022taq} claims that
every $W(E_8)$-invariant weak Jacobi form is expressed
uniquely as a polynomial of 11 meromorphic
$W(E_8)$-invariant weak Jacobi forms over $\bbC$.
Its proof relies on the canonical expression
for $W(E_8)$-invariant weak Jacobi forms
formulated in \cite[Theorem 1.1]{Sun:2021ije}.
In our proof here, Lemma~\ref{lem:RtriDinv} plays the role of
\cite[Theorem 1.1]{Sun:2021ije}.}
Let $\varphi$ be a $D_4$ triality invariant
of some fixed weight and degree.
By Lemma~\ref{lem:RtriDinv}
there exists $n\in\bbZ_{\ge 0}$ such that
$\Delta^n\varphi$ is expressed as a polynomial
of $\Delta K,\Delta^2L,\Delta^2M,\Delta^3N$ over $M_*$.
As we just saw in \eqref{eq:KLMNinab},
these are polynomials of $a_i,b_j$.
Therefore, $\Delta^n\varphi$ is also expressed
as a polynomial $P$ of $a_i,b_j$:
\begin{align}
\Delta^n\varphi=P(a_0,a_2,b_0,b_1,b_2,b_3)\in\Rab.
\end{align}
If $n=0$, then we are done, so assume that $n>0$.
Since $\Delta=q+O(q^2)$ and $\Sop(\varphi)$ is
a formal power series in $q$,
the constant term of $\Sop(P)$ has to vanish.
By Proposition~\ref{prop:abprop} (3),
the constant term of $\Sop(P)$ is a polynomial of
$a_0^{(0)},a_2^{(0)},b_0^{(0)},b_1^{(0)},b_2^{(0)},b_3^{(0)}$
given in \eqref{eq:abq0form},
which now form an algebraic relation.
Since $a_2^{(0)},b_1^{(0)},b_2^{(0)},b_3^{(0)}$
are algebraically independent by Proposition~\ref{prop:abcdq0indep}
and are of nonzero degree,
the only possible algebraic relation
compatible with the bigrading is of the form
$\bigl(a_0^{(0)}\bigr)^3-27\bigl(b_0^{(0)}\bigr)^2$
multiplied by some polynomial of $a_i^{(0)},b_j^{(0)}$.
This means that the polynomial $P$ is divisible by
$\Delta=a_0^3-27b_0^2$, i.e.~it is written as
$P=\Delta P_1$ with $P_1\in\Rab$.

Now we have $\Delta^{n-1}\varphi=P_1$.
We can repeat the above discussion and inductively show that 
$P=\Delta P_1=\Delta^2 P_2=\cdots=\Delta^n P_n$ with $P_n\in\Rab$,
which gives $\varphi$.
\end{proof}
In the same way, we can prove the following lemma.
\begin{lem}
The ring $\Rtri_{*,*}$ of $D_4$ triality invariants
is contained in $\Rcd$:
\begin{align}
\Rtri_{*,*}\subset \Rcd.
\end{align}
In other words, every $D_4$ triality invariant
is expressed uniquely as a polynomial of
$c_0,c_1,c_2,d_0,d_2,d_3$ over $\bbC$. 
\end{lem}

We are now ready to state the main theorem in this subsection.
\begin{thm}\label{thm:RcaptR}
The ring $\Rtri_{*,*}$ of $D_4$ triality invariants
is the intersection of two polynomial rings $\Rab$ and $\Rcd$
\begin{align}
\Rtri_{*,*}=\Rab\cap\Rcd.
\end{align}
In other words, the following two conditions
for a function $\varphi$ are equivalent:
\renewcommand{\theenumi}{\roman{enumi}}
\renewcommand{\labelenumi}{$\mathrm{(\theenumi})$}
~\\[-3.5ex]
\begin{enumerate}
\item $\varphi$ is a $D_4$ triality invariant.

\vspace{-1ex}
\item $\varphi$ is
a polynomial of $a_i,b_j$ and also a polynomial of $c_i,d_j$.
\end{enumerate}
\end{thm}
\begin{proof}
By the last two lemmas it is clear that
$\Rtri_{*,*}\subset\Rab\cap\Rcd$.
Therefore, suppose that $\varphi\in\Rab\cap\Rcd$.
Then $\varphi$ satisfies all the conditions of being
a $D_4$ triality invariant, except that it could have poles
at the zero points of $E_4$ and $E_6$.
However, by Proposition~\ref{prop:abprop} (2),
$\varphi$ is holomorphic at the zero points of $E_6$.
Similarly, by Proposition~\ref{prop:cdprop} (2),
$\varphi$ is holomorphic at the zero points of $E_4$.
\end{proof}

As an application of Theorem~\ref{thm:RcaptR}
we formulate an algorithm for constructing
all $D_4$ triality invariants of given weight $k$ and degree $m$.
\begin{algo}\label{algo:D4}
The algorithm for constructing all $D_4$ triality invariants of given
weight $k$ and degree $m$ goes as follows.
\renewcommand{\theenumi}{\arabic{enumi}}
\renewcommand{\labelenumi}{(\theenumi)}
\begin{enumerate}
\item Take the most general polynomial in $a_i,b_j$ of
weight $k$ and degree $m$ as our ansatz. More specifically,
the ansatz is constructed as the most general linear combination
of all monomials of $a_i,b_j$
appearing in the coefficient of $x^ky^m$ in the generating series
\begin{align}
\frac{1}{\prod_{i=0,2}(1-x^{4+2i}y^{2i}a_i)
         \prod_{j=0}^3(1-x^{6+2j}y^{2j}b_j)}.
\end{align}

\item Substitute \eqref{eq:abincd} into the ansatz and solve 
the linear equations among undetermined coefficients so that
all negative powers of $c_0$ vanish.

\item Substitute the general solution back into the original ansatz.
This gives the most general linear combination of $D_4$ triality
invariants of weight $k$ and degree $m$.
\end{enumerate}
\end{algo}
%

\section{Classical invariant theory and isomorphism}\label{sec:iso}

In this section we prove the main theorem of this paper.
We give only a minimal description
since the discussion below is basically the same as
that in \cite[Section~4]{Sakai:2024vby}.

\subsection{Classical invariant theory}

In this subsection we present some definitions and
useful results of classical invariant theory.

Let $f_k\ (k=1,\ldots,r)$ be binary forms of degree $n_k$.
We write them as
\begin{align}
f_k(u,v)=\sum_{i=0}^{n_k}\alpha_{k,i} u^{n_k-i}v^i
\end{align}
with $\alpha_{k,i}\in\bbC$, $(u,v)\in\bbC^2$.
Define the action of $\grp{SL}_2(\bbC)$ on the variables $u,v$ by
\begin{align}
(u',v')=(t_{11}u+t_{12}v,t_{21}u+t_{22}v),\qquad
T=\begin{pmatrix} t_{11}&t_{12}\\ t_{21}&t_{22}\end{pmatrix}
\in\grp{SL}_2(\bbC).
\end{align}
We let the symbol $({}')$ denote the action of $T$:
$(u',v')=T(u,v),\,\alpha'_i=T\alpha_i$, etc.
The action of $T$ on $f_k$
and its coefficients $\alpha_{k,i}$ is then defined by
\begin{align}
f'_k(u,v)=\sum_{i=0}^{n_k}\alpha'_{k,i} u^{{n_k}-i}v^i:=f_k(u',v').
\end{align}
\begin{dfn}[Invariants of binary forms]
A joint invariant of $f_1,\ldots,f_r$ is a homogeneous
polynomial $\Phi$ in the coefficients $\alpha_{k,i}$ of $f_k$
that satisfies
\begin{align}
\Phi(\alpha'_{k,i})=\Phi(\alpha_{k,i})
\end{align}
under the action of all $T\in\grp{SL}_2(\bbC)$.
The degree of $\Phi$ refers to the standard degree of
homogeneity in $\alpha_{k,i}$.
\end{dfn}
\begin{dfn}[Covariants of binary forms]\label{def:cov}
A joint covariant of $f_1,\ldots,f_r$ of degree $d$ and
order $\omega$ is a polynomial $\Psi(\alpha_{k,i};u,v)$ that
satisfies the following conditions:
\renewcommand{\theenumi}{\roman{enumi}}
\renewcommand{\labelenumi}{(\theenumi)}
~\\[-3ex]
\begin{enumerate}
\item $\Psi$ is homogeneous of degree $d$ in the coefficients
$\alpha_{k,i}$.

\item $\Psi$ is homogeneous of degree $\omega$ in the variables
$u,v$.

\item $\Psi$ satisfies
\begin{align}
\Psi(\alpha'_{k,i};u,v)=\Psi(\alpha_{k,i};u',v')
\end{align}
under the action of all $T\in\grp{SL}_2(\bbC)$.
\end{enumerate}
\end{dfn}
\begin{dfn}[Semiinvariants of binary forms]\label{def:semi}
A joint semiinvariant of $f_1,\ldots,f_r$ of degree $d$ and
order $\omega$ is a polynomial $\Phi(\alpha_{k,i})$
that satisfies the following conditions:
\renewcommand{\theenumi}{\roman{enumi}}
\renewcommand{\labelenumi}{(\theenumi)}
~\\[-7ex]
\begin{enumerate}
\item $\Phi$ is homogeneous of degree $d$ in the coefficients
$\alpha_{k,i}$.

\item $\Phi$ satisfies
\begin{align}
\Phi(\alpha'_{k,i})=\Phi(\alpha_{k,i})
\end{align}
under the action of
all $T=\begin{pmatrix}1&\kappa\\ 0&1\end{pmatrix}\in\grp{SL}_2(\bbC)$.

\item $\Phi$ transforms as
\begin{align}
\Phi(\alpha'_{k,i})=\lambda^\omega\Phi(\alpha_{k,i})
\end{align}
under the action of
all
$T=\begin{pmatrix}\lambda&0\\ 0&\lambda^{-1}\end{pmatrix}\in
 \grp{SL}_2(\bbC)$.
\end{enumerate}
\end{dfn}
We often omit `joint' when it does not cause any confusion.

Let $V_n$ denote the $\bbC$-vector space of binary forms of degree $n$.
The rings of invariants, covariants, semiinvariants of $f_1,\ldots,f_r$
are expressed respectively as
\begin{align}
\bbC[V_{n_1}\oplus\cdots\oplus V_{n_r}]^{\grp{SL}_2(\bbC)},\quad
\bbC[V_{n_1}\oplus\cdots\oplus V_{n_r}\oplus\bbC^2]^{\grp{SL}_2(\bbC)},
\quad
\bbC[V_{n_1}\oplus\cdots\oplus V_{n_r}]^{\grp{U}_2(\bbC)}.
\end{align}
Here, $\grp{U}_2(\bbC)$ is the group of
upper triangular unipotent matrices
\begin{align}
\grp{U}_2(\bbC)
 :=\left\{\begin{pmatrix}1&\kappa\\ 0&1\end{pmatrix}
 \,\middle|\,\kappa\in\bbC\right\}.
\end{align}

A covariant $\Psi$ of order $\omega$ can be expanded as
\begin{align}
\Psi(\alpha_{k,i};u,v)
=\Psi_0(\alpha_{k,i}) u^\omega
 +\Psi_1(\alpha_{k,i}) u^{\omega-1}v
 +\cdots +\Psi_\omega(\alpha_{k,i}) v^\omega.
\end{align}
It is easy to see that the leading coefficient
\begin{align}
\Psi_0(\alpha_{k,i})=\Psi(\alpha_{k,i};1,0)
\end{align}
is a semiinvariant. The converse is also true:
\begin{thm}[{Roberts isomorphism
            \cite{Roberts:1861}}]\label{thm:Roberts}
There is an isomorphism
\begin{align}
\begin{aligned}
\bbC[V_{n_1}\oplus\cdots\oplus V_{n_r}
     \oplus\bbC^2]^{\grp{SL}_2(\bbC)}
&\ \stackrel\sim\longrightarrow\ 
\bbC[V_{n_1}\oplus\cdots\oplus V_{n_r}]^{\grp{U}_2(\bbC)};\\
\Psi(\alpha_{k,i};u,v)&\ \ \mapsto\ \
\Psi_0(\alpha_{k,i}).
\end{aligned}
\end{align}
It maps a covariant to a semiinvariant
of the same degree and order.
\end{thm}
The inverse mapping is also easily constructed.
Given a semiinvariant $\Psi_0$ of order $\omega$,
the corresponding covariant $\Psi$ is given by
\begin{align}\label{eq:Robinv}
\Psi(\alpha_{k,i};u,v)
=u^\omega\Psi_0(\halpha_{k,i}(u,v))
\end{align}
with
\begin{align}
\halpha_{k,i}(u,v)
=\sum_{j=i}^{n_k}\alpha_{k,j}
 \begin{pmatrix}j\\[1ex] i\end{pmatrix}
 \left(\frac{v}{u}\right)^{j-i}.
\end{align}

For later use, let us define functions $\gamma_{k,i}^{(m)}$
of $\alpha_{l,j}$ by the relations
\begin{align}
\sum_{i=0}^{n_k}\gamma_{k,i}^{(m)}u^{n_k-i}
=f_k\left(u-\frac{\alpha_{m,1}}{n_m\alpha_{m,0}},1\right).
\label{eq:alkidef}
\end{align}
Explicitly, $\gamma_{k,i}^{(m)}$ are given by
\begin{align}
\gamma_{k,i}^{(m)}
=\sum_{j=0}^i\alpha_{k,j}
 \begin{pmatrix} n_i-j\\[1ex] n_i-i\end{pmatrix}
 \left(-\frac{\alpha_{m,1}}{n_m\alpha_{m,0}}\right)^{i-j}.
\end{align}
Note that by construction we have $\gamma_{k,1}^{(k)}=0$.
The following fact is well known in classical invariant theory
(see e.g.~\cite[Section~8.1.7]{Glenn:1915}).
\begin{prop}\label{prop:semiex}
The polynomials $\alpha_{k,0}^{i-1}\gamma_{k,i}^{(k)}$ are
semiinvariants of $f_k$. The polynomials
$\alpha_{m,0}^i\gamma_{k,i}^{(m)}\ (k\ne m)$ are joint
semiinvariants of $f_k$ and $f_m$.
\end{prop}
\begin{proof}
An elementary proof is given in the author's previous work
\cite{Sakai:2024vby}.
\end{proof}
Note that $\alpha_{k,i}$ and $\gamma_{k,i}^{(m)}$
scale under the transformation $(u',v')=(\lambda u,\lambda^{-1}v)$
as
\begin{align}
\alpha'_{k,i}=\lambda^{n_k-2i}\alpha_{k,i},\qquad
{\gamma'}_{k,i}^{(m)}=\lambda^{n_k-2i}\gamma_{k,i}^{(m)},
\label{eq:alorder}
\end{align}
which we will use later.

\subsection{Trigrading}

In this subsection we introduce a trigrading to
the ring $\Rtri_{*,*}$ of $D_4$ triality invariants.
We will need it in constructing our isomorphism
in the next subsection.

Let $f$ be a binary quadratic and $g$ a binary cubic
\begin{align}
f=\sum_{i=0}^2\alpha_iu^{2-i}v^i,\qquad
g=\sum_{i=0}^3\beta_iu^{3-i}v^i.
\end{align}
We are interested in joint covariants $\Psi(\alpha_i,\beta_j;u,v)$
and joint semiinvariants $\Phi(\alpha_i,\beta_j)$ of $f$ and $g$.
The rings of them are denoted respectively by
\begin{align}
\bbC[V_2\oplus V_3\oplus\bbC^2]^{\grp{SL_2(\bbC)}},\qquad
\bbC[V_2\oplus V_3]^{\grp{U}_2(\bbC)}.
\end{align}
Since we always deal with $f$ and $g$ of fixed degrees,
by abuse of notation we often omit the variable $v$
and write them as polynomials of $u$
\begin{align}
f(u,1)=f(u),\qquad g(u,1)=g(u).
\end{align}

Let us introduce the notion of refined degrees.
\begin{dfn}\label{def:degalbe}
We say that $\Xi\in
 \bbC[\alpha_0,\alpha_1,\alpha_2,\beta_0,\ldots,\beta_3,
 \alpha_0^{-1},\beta_0^{-1},u,v]$
is of degrees $(d_\alpha,d_\beta)$ if $\Xi$ is homogeneous of
degree $d_\alpha$ in $\alpha_i$ and homogeneous of
degree $d_\beta$ in $\beta_i$.
\end{dfn}
\begin{rem}
A covariant (semiinvariant) of degrees $(d_\alpha,d_\beta)$ is
of degree $d=d_\alpha+d_\beta$ in the sense of 
Definition~\ref{def:cov} (Definition~\ref{def:semi}).
\end{rem}
The ring $\bbC[V_2\oplus V_3\oplus\bbC^2]^{\grp{SL_2(\bbC)}}$ of
covariants and the ring $\bbC[V_2\oplus V_3]^{\grp{U}_2(\bbC)}$ of
semiinvariants admit a trigrading, graded by
degrees $d_\alpha,d_\beta$ and order $\omega$.

Similar refined degrees can be introduced
for (weak) $D_4$ triality invariants:
\begin{dfn}\label{def:refdegab}
We say that
$\xi\in\bbC[a_0,a_2,b_0,b_1,b_2,b_3,a_0^{-1},b_0^{-1}]$ is of
degrees $(d_a,d_b)$ if $\xi$ is homogeneous of degree $d_a$ in $a_i$
and homogeneous of degree $d_b$ in $b_j$.
\end{dfn}
In the same way,
we can introduce refined degrees $(d_c,d_d)$ in terms of $c_i,d_j$ to
\begin{align}
 \bbC[c_0,c_1,c_2,d_0,d_2,d_3,c_0^{-1},d_0^{-1}]
=\bbC[a_0,a_2,b_0,b_1,b_2,b_3,a_0^{-1},b_0^{-1}].
\label{eq:abcdring}
\end{align}
\begin{prop}
The degrees $(d_c,d_d)$ are identical with $(d_a,d_b)$.
\end{prop}
\begin{proof}
It is clear from the relations \eqref{eq:abincd}
between $a_i,b_j$ and $c_k,d_l$.
\end{proof}
Introduction of these degrees to the ring $\Rtri_{*,*}$ leads
to an interesting result:
\begin{thm}\label{thm:trigrad}
$D_4$ triality invariants form a trigraded ring over $\bbC$,
graded by refined degrees $d_a,d_b$ and polynomial degree $m$
(in the sense of Definition~\ref{def:triinv}):
\begin{align}
\Rtri_{*,*}=\bigoplus_{d_a,d_b,m=0}^\infty\Rtri_{d_a,d_b,m}.
\end{align}
Every element of $\Rtri_{d_a,d_b,m}$ is of weight
\begin{align}
k=4d_a+6d_b+m.
\label{eq:graderel}
\end{align}
\end{thm}
\begin{proof}
By Lemma~\ref{lem:RGinRab}, the gradation by degrees $(d_a,d_b)$ is
consistently introduced to the ring $\Rtri_{*,*}$: every element of
$\Rtri_{*,*}$ is expressed uniquely as a polynomial of $a_i,b_j$
and the degrees can be determined term-by-term.
The functions $a_i$ and $b_j$ are of degrees $(1,0)$ and $(0,1)$,
respectively, and have weight and polynomial degree described
in Proposition~\ref{prop:abprop} (1). The degrees $(d_a,d_b)$,
weight $k$ and polynomial degree $m$ of $a_i,b_j$ are subject
to the relation \eqref{eq:graderel}.
\end{proof}
%

\subsection{Isomorphism}\label{subsec:iso}

In this subsection we prove Theorem~\ref{thm:RtriandS},
or equivalently Theorem~\ref{thm:RtriandCov}, which is
the main theorem of this paper. We do this by constructing
an explicit isomorphism.

Let us define functions $\ha_i,\hb_i,\hc_i,\hd_i$
of $\alpha_j,\beta_k$ by the relations
\begin{align}
\begin{aligned}
\sum_{i=0}^2\ha_i u^{2-i}
&=\sum_{i=0}^2\alpha_i
 \left(u-\frac{\alpha_1}{2\alpha_0}\right)^{2-i},\quad&
\sum_{i=0}^3\hb_i u^{3-i}
&=\sum_{i=0}^3\beta_i
 \left(u-\frac{\alpha_1}{2\alpha_0}\right)^{3-i},\\
\sum_{i=0}^2\hc_i \tu^{2-i}
&=\sum_{i=0}^2\alpha_i
 \left(\tu-\frac{\beta_1}{3\beta_0}\right)^{2-i},&
\sum_{i=0}^3\hd_i \tu^{3-i}
&=\sum_{i=0}^3\beta_i
 \left(\tu-\frac{\beta_1}{3\beta_0}\right)^{3-i}.
\end{aligned}
\label{eq:habcddef}
\end{align}
Explicitly, they are given by
\begin{align}
\begin{aligned}
\ha_i&=\sum_{j=0}^i\alpha_j
 \begin{pmatrix} 2-j\\[1ex] 2-i\end{pmatrix}
 \left(-\frac{\alpha_1}{2\alpha_0}\right)^{i-j},\quad&&
\hb_i=\sum_{j=0}^i\beta_j
 \begin{pmatrix} 3-j\\[1ex] 3-i\end{pmatrix}
 \left(-\frac{\alpha_1}{2\alpha_0}\right)^{i-j},\\
\hc_i&=\sum_{j=0}^i\alpha_j
 \begin{pmatrix} 2-j\\[1ex] 2-i\end{pmatrix}
 \left(-\frac{\beta_1}{3\beta_0}\right)^{i-j},&&
\hd_i=\sum_{j=0}^i\beta_j
 \begin{pmatrix} 3-j\\[1ex] 3-i\end{pmatrix}
 \left(-\frac{\beta_1}{3\beta_0}\right)^{i-j}.
\end{aligned}
\label{eq:habcdform}
\end{align}
Note that by construction we have
\begin{align}
\ha_1=0,\quad\hd_1=0.
\label{eq:ha1hd1}
\end{align}
\begin{lem}\label{lem:abcdrel}
$\ha_i,\hb_i,\hc_i,\hd_i$ satisfy the same relations
as \eqref{eq:abincd} for $a_i,b_i,c_i,d_i$.
\end{lem}
\begin{proof}
Since $u$ and $\tu$ in \eqref{eq:habcddef} are isolated
formal variables, we are allowed to rewrite them
in terms of a new variable $\breve{u}$ as
\begin{align}
 u-\frac{\alpha_1}{2\alpha_0}
=\tu-\frac{\beta_1}{3\beta_0}=\breve{u}.
\end{align}
This leads to the relations
\begin{align}
\sum_{i=0}^2\ha_i u^{2-i}=\sum_{i=0}^2\hc_i \tu^{2-i},\qquad
\sum_{i=0}^3\hb_i u^{3-i}=\sum_{i=0}^3\hd_i \tu^{3-i}.
\end{align}
Combining these with \eqref{eq:ha1hd1}, we obtain
the same relations as in \eqref{eq:abincd}.
\end{proof}
\begin{lem}\label{lem:abcdsemi}
$\alpha_0^{i-1}\ha_i,\ \alpha_0^i\hb_i,\
 \beta_0^i\hc_i,\ \beta_0^{i-1}\hd_i$
are joint semiinvariants of $f$ and $g$.
\end{lem}
\begin{proof}
This is a special case of Proposition~\ref{prop:semiex}.
\end{proof}

By the substitution
$a_i=\ha_i,b_i=\hb_i,c_i=\hc_i,d_i=\hd_i$ let us introduce
the mappings
\begin{align}
\begin{aligned}
\psi&:
\Rab=\bbC[a_0,a_2,b_0,b_1,b_2,b_3]
\to
\bbC[\alpha_0,\alpha_1,\alpha_2,
\beta_0,\beta_1,\beta_2,\beta_3,\alpha_0^{-1}],\\
\tpsi&:
\Rcd=\bbC[c_0,c_1,c_2,d_0,d_2,d_3]
\to
\bbC[\alpha_0,\alpha_1,\alpha_2,
\beta_0,\beta_1,\beta_2,\beta_3,\beta_0^{-1}].
\end{aligned}
\end{align}
It follows from the explicit expressions \eqref{eq:habcdform}
that the following lemma holds:
\begin{lem}\label{lem:psiinjec}
The mappings $\psi$ and $\tpsi$ are injective. The inverse mappings
from their images are given by the substitutions
\begin{align}
\begin{aligned}
\psi^{-1}&:\ 
\alpha_i=a_i\ (i\ne 1),\quad \alpha_1=0,\quad
\beta_j=b_j,\\
\tpsi^{-1}&:\ 
\alpha_i=c_i,\quad
\beta_j=d_j\ (j\ne 1),\quad \beta_1=0.
\end{aligned}
\end{align}
These mappings identify
the degrees $(d_a,d_b)$ with $(d_\alpha,d_\beta)$.
\end{lem}
Recall that the ring of $D_4$ triality invariants
is described as $\Rtri_{*,*}=\Rab\cap\Rcd$
by Theorem~\ref{thm:RcaptR}.
\begin{lem}\label{lem:hpsimap}
The restrictions of $\psi$ and $\tpsi$
to $\Rtri_{*,*}=\Rab\cap\Rcd$ are identical
\begin{align}
\hpsi:=\psi\big|_{\Rtri_{*,*}}=\tpsi\big|_{\Rtri_{*,*}}
\end{align}
and define an injection
\begin{align}
\hpsi:\Rtri_{*,*}
\to\bbC[\alpha_0,\alpha_1,\alpha_2,\beta_0,\beta_1,\beta_2,\beta_3].
\end{align}
This injection
identifies the degrees $(d_a,d_b)$ with $(d_\alpha,d_\beta)$.
\end{lem}
\begin{proof}
The restrictions $\psi\big|_{\Rtri_{*,*}}$ and
$\tpsi\big|_{\Rtri_{*,*}}$ are identical by Lemma~\ref{lem:abcdrel}.
The intersection of the codomains of $\psi$ and $\tpsi$ is
$\bbC[\alpha_0,\alpha_1,\alpha_2,\beta_0,\beta_1,\beta_2,\beta_3]$.
\end{proof}
We are now ready to prove the main theorem of this paper.
\begin{thm}\label{thm:RtriandS}
The injection $\hpsi$ gives an isomorphism between the ring of
$D_4$ triality invariants and the ring of joint semiinvariants
of $f$ and $g$:
\begin{align}
\begin{aligned}
\Rtri_{*,*}&\ \cong\ \bbC[V_2\oplus V_3]^{\grp{U}_2(\bbC)};\\
\hpsi:\varphi(a_i,b_j)&\ \mapsto\ \varphi(\ha_i,\hb_j),\\
\hpsi^{-1}:\Phi(a_i,b_j)&\ \mapsfrom\ \Phi(\alpha_i,\beta_j).
\end{aligned}
\end{align}
Every $D_4$ triality invariant of degrees $(d_a,d_b)$, weight $k$ 
and polynomial degree $m$ is mapped to a semiinvariant of
the same degrees $(d_a,d_b)$ and order $\omega=(k-3m)/2$.
\end{thm}
\begin{proof}
Suppose that $\varphi(a_i,b_j)$ is a $D_4$ triality invariant
of degrees $(d_a,d_b,m)$ and weight $k$. By Lemma~\ref{lem:hpsimap},
$\hpsi(\varphi)=\varphi(\ha_i,\hb_j)$ is a polynomial that is
homogeneous of degree $d_a$ in $\alpha_i$ and homogeneous of
degree $d_b$ in $\beta_j$. By Lemma~\ref{lem:abcdsemi},
$\varphi(\ha_i,\hb_j)$ is invariant under the transformation
$(u',v')=(u+\kappa v,v)$. (Note that $\alpha_0,\beta_0$
are joint semiinvariants.) To see the scaling behavior
under the transformation $(u',v')=(\lambda u,\lambda^{-1}v)$,
recall that $\ha_i,\hb_j$ scale as
$\ha'_i=\lambda^{2-2i}\ha_i$, $\hb'_i=\lambda^{3-2i}\hb_i$
(which are special cases of \eqref{eq:alorder}). This implies that
$\varphi(\ha_i,\hb_j)$ scales as
\begin{align}
\varphi(\ha'_i,\hb'_j)
=\lambda^{2d_a+3d_b-m}\varphi(\ha_i,\hb_j).
\end{align}
(Recall that the subscript $i$ of $a_i,b_i$ represents half their
polynomial degree in $z_j$.) Therefore, $\varphi(\ha_i,\hb_j)$ is
a semiinvariant of $f$ and $g$ of degrees $(d_a,d_b)$ and
order $\omega=2d_a+3d_b-m$, which is also written as
$\omega=(k-3m)/2$ by Theorem~\ref{thm:trigrad}.

Conversely, suppose that $\Phi(\alpha_i,\beta_j)$ is a joint
semiinvariant of $f$ and $g$. Then $\Phi$ is invariant under the
translation of $u$ that transforms the curve \eqref{eq:abcurve} to
the curve \eqref{eq:cdcurve} and thus $\Phi(a_i,b_j)=\Phi(c_i,d_j)$.
Clearly, this is a polynomial of $a_i,b_j$ and, at the same time, 
a polynomial of $c_i,d_j$. By Theorem~\ref{thm:RcaptR}, this is
a $D_4$ triality invariant. By Lemma~\ref{lem:psiinjec}, this gives
the inverse mapping
$\hpsi^{-1}(\Phi(\alpha_i,\beta_j)) = \Phi(a_i,b_j)=\Phi(c_i,d_j)$.

The bijection obviously preserves the structure of the rings and thus
gives a ring isomorphism.
\end{proof}

Recall that the ring of semiinvariants is isomorphic to
the ring covariants (Roberts isomorphism, Theorem~\ref{thm:Roberts}).
We can therefore rephrase Theorem~\ref{thm:RtriandS} as follows:
\begin{thm}\label{thm:RtriandCov}
The composite of the mappings \eqref{eq:Robinv} and $\hpsi$ gives 
an isomorphism between the ring of $D_4$ triality invariants and
the ring of joint covariants of $f$ and $g$:
\begin{align}
\Rtri_{*,*}
\ \cong\ \bbC[V_2\oplus V_3\oplus\bbC^2]^{\grp{SL_2(\bbC)}}.
\end{align}
\end{thm}
\begin{rem}\label{rem:graderel}
It is helpful to summarize the relations among various grades
that we have introduced:
the weight $k$ and polynomial degree $m$ of $D_4$ triality invariants,
the degree $d$ and order $\omega$ of semiinvariants
(covariants),
and the common degrees $(d_a,d_b)=(d_c,d_d)=(d_\alpha,d_\beta)$.
They are related as
\begin{align}
\begin{aligned}
d=d_a+d_b,\qquad
\omega=\frac{k-3m}{2}=2d_a+3d_b-m,
\end{aligned}
\end{align}
or conversely,
\begin{align}
\begin{aligned}
k=4d_a+6d_b+m
 =6d_a+9d_b-\omega,\qquad
m=2d_a+3d_b-\omega.
\end{aligned}
\end{align}
\end{rem}
Since there are no covariants of negative order,
we immediately obtain the following corollary
of Theorem~\ref{thm:RtriandCov}:
\begin{cor}\label{cor:lowerbound}
The weight of nonzero $D_4$ triality invariants
of degree $m$ is not less than $3m$.
\end{cor}
%

\section{Generators}\label{sec:gen}

\subsection{Minimal basis for the covariant ring}

In this subsection we recall a minimal basis for
the covariant ring 
$\bbC[V_2\oplus V_3\oplus\bbC^2]^{\grp{SL_2(\bbC)}}$.

Finding a minimal basis for invariant/covariant rings
is a classical problem.
A common approach is to first construct
a complete (but not necessarily minimal) basis
and then eliminate reducible generators to remove the redundancy.
In principle, this is always a feasible combinatorial problem,
due to the Hilbert's basis theorem.
In practice, however, the computation soon becomes
involved as the degrees of binary forms grow.
There is a long history of developments in computation.
See \cite{Olive:2014} for recent advances in algorithms
and for a summary of the history.

In the present case, 
a minimal basis of generators for the covariant ring
$\bbC[V_2\oplus V_3\oplus\bbC^2]^{\grp{SL_2(\bbC)}}$
has been known since the 19th century
\cite[Abs\"{a}tze 296--299]{Gordan:1887}.
Nevertheless, since this is a simple yet nontrivial example,
it is instructive to demonstrate how to construct
a complete basis and how to reduce it to a minimal one.
In what follows
we will take a traditional approach \cite{Grace:1903}
to construct a basis and then check its minimality,
with an emphasis on the combinatorial aspects of the problem.

Let $F(u,v), G(u,v)$ be binary forms of degree $m, n$,
respectively. For any $r\in\bbZ_{\ge 0}$, the $r$th transvectant
of $F$ and $G$ is defined as
\begin{align}
\tv{F}{G}{r}
:=\frac{(m-r)!(n-r)!}{m!n!}
\sum_{j=0}^r(-1)^j\begin{pmatrix}r\\ j\end{pmatrix}
\frac{\partial^r F}{\partial u^{r-j}\partial v^j}
\frac{\partial^r G}{\partial u^j\partial v^{r-j}}.
\label{eq:FGr}
\end{align}
Binary forms and their transvectants
have a concrete realization in terms of
tensor contractions, which leads to
a convenient diagrammatic notation
as explained below
(for the details, see e.g.~\cite{Abdesselam:2012, Abdesselam:2024}).
A binary form $F(u,v)$ of degree $m$ can be written as
\begin{align}
F(u,v)
 =\sum_{i_1,\ldots,i_m=1}^2 F_{i_1\cdots i_m}u_{i_1}\cdots u_{i_m},
\label{eq:Fuv}
\end{align}
where $F_{i_1,\ldots,i_m}$ is a symmetric tensor
of rank $m$ and $u_i$ is a vector whose components are
\begin{align}
u_1=u,\quad u_2=v.
\end{align}
These tensors and the binary form \eqref{eq:Fuv} are
diagrammatically expressed as
\begin{align}
F_{i_1i_2\cdots i_m}=
\raisebox{-3.8ex}{
\begin{tikzpicture}
\draw[line width=1pt](0,0)--(-.39,-.46);
\draw[line width=1pt](0,0)--(-.10,-.59);
\draw[line width=1pt](0,0)--(.39,-.46);
\draw[fill=black!10,line width=1pt](0,0)circle(.4);
\node at (0,0){$F$};
\node at (-.48,-.57){\scriptsize $i_1$};
\node at (-.13,-.74){\scriptsize $i_2$};
\node at (.60,-.57){\scriptsize $i_m$};
\filldraw(.026,-.6)circle(.01);
\filldraw(.155,-.58)circle(.01);
\filldraw(.277,-.532)circle(.01);
\end{tikzpicture}
},
\qquad
u_i=
\raisebox{-3ex}{
\begin{tikzpicture}
\draw[line width=1pt](0,0)--(0,-.4);
\draw[fill=black!10,line width=1pt](0,0)circle(.2);
\node at (0,0){$u$};
\node at (0,-.6){\scriptsize$i$};
\end{tikzpicture}
},\qquad
F(u,v)=
\raisebox{-4ex}{
\begin{tikzpicture}
\draw[line width=1pt](0,0)--(-.64,-.77);
\draw[line width=1pt](0,0)--(-.17,-.98);
\draw[line width=1pt](0,0)--(.64,-.77);
\draw[fill=black!10,line width=1pt](0,0)circle(.4);
\node at (0,0){$F$};
\draw[fill=black!10,line width=1pt](-.64,-.77)circle(.2);
\node at (-.64,-.77){$u$};
\draw[fill=black!10,line width=1pt](-.17,-.98)circle(.2);
\node at (-.17,-.98){$u$};
\draw[fill=black!10,line width=1pt](.64,-.77)circle(.2);
\node at (.64,-.77){$u$};
\filldraw(.026,-.6)circle(.01);
\filldraw(.155,-.58)circle(.01);
\filldraw(.277,-.532)circle(.01);
\end{tikzpicture}
}.
\label{eq:micro1}
\end{align}
As one can see, a rank-$m$ symmetric tensor is denoted
by a round node with $m$ legs
and contracted indices are represented by connected legs.
The transvectant \eqref{eq:FGr} is then expressed as
\begin{align}
\tv{F}{G}{r}=
\raisebox{-4ex}{
\begin{tikzpicture}
\draw[line width=1pt](0,-.15)--(1.4,-.15);
\draw[line width=1pt](0,.15)--(1.4,.15);
\draw[line width=.8pt,-Stealth](.60,-.15)--(.55,-.15);
\draw[line width=.85pt,-Stealth](.60,.15)--(.55,.15);
\filldraw(.5,-.075)circle(.01);
\filldraw(.5,0)circle(.01);
\filldraw(.5,.075)circle(.01);
\filldraw(.9,-.075)circle(.01);
\filldraw(.9,0)circle(.01);
\filldraw(.9,.075)circle(.01);
\draw[line width=1pt](0,0)--(-.342,-.94);
\draw[line width=1pt](0,0)--(.342,-.94);
\draw[fill=black!10,line width=1pt](0,0)circle(.4);
\node at (0,0){$F$};
\draw[fill=black!10,line width=1pt](-.342,-.94)circle(.2);
\node at (-.342,-.94){$u$};
\draw[fill=black!10,line width=1pt](.342,-.94)circle(.2);
\node at (.342,-.94){$u$};
\filldraw(-.104,-.591)circle(.01);
\filldraw(0,-.6)circle(.01);
\filldraw(.104,-.591)circle(.01);
\draw[line width=1pt](1.4,0)--(1.058,-.94);
\draw[line width=1pt](1.4,0)--(1.742,-.94);
\draw[fill=black!10,line width=1pt](1.4,0)circle(.4);
\node at (1.4,0){$G$};
\draw[fill=black!10,line width=1pt](1.058,-.94)circle(.2);
\node at (1.058,-.94){$u$};
\draw[fill=black!10,line width=1pt](1.742,-.94)circle(.2);
\node at (1.742,-.94){$u$};
\filldraw(1.296,-.591)circle(.01);
\filldraw(1.4,-.6)circle(.01);
\filldraw(1.504,-.591)circle(.01);
\node at (.7,.38){{\scriptsize$r$}};
\draw (-.17,-.52) node [left]{{\scriptsize$m-r$}};
\draw (1.6,.-.52) node [right]{{\scriptsize$n-r$}};
\end{tikzpicture}
}.
\end{align}
Here, an arrow represents the antisymmetric tensor
\begin{align}
\raisebox{-.85ex}{
\begin{tikzpicture}
\draw[line width=.8pt,-Stealth](.40,0)--(.35,0);
\draw[line width=1pt](0,0)--(1,0);
\draw (0,0) node [left]{{\scriptsize $\phantom{j}i$}};
\draw (1,0) node [right]{{\scriptsize $j\phantom{i}$}};
\end{tikzpicture}
}
\!\!=\ \epsilon_{ij}
\label{eq:micro3}
\end{align}
with 
$\epsilon_{12}=-\epsilon_{21}=1,\ \epsilon_{11}=\epsilon_{22}=0$.
The diagrammatic notation not only gives a concrete realization
of transvectants, but also is very helpful
in counting the degree and order of covariants,
as we will see later.

As in the last section,
we let $f(u,v)$ be a binary quadratic and $g(u,v)$ a binary cubic.
These forms themselves are covariants in the ring
$\bbC[V_2\oplus V_3\oplus\bbC^2]^{\grp{SL_2(\bbC)}}$.
As is well known in classical invariant theory,
every covariant in this ring is constructed
as a (nested) transvectant of $f$ and $g$.
To select relevant basis elements,
we start from the following well-known facts:
\renewcommand{\theenumi}{\arabic{enumi}}
\renewcommand{\labelenumi}{$(\theenumi)$}
\begin{enumerate}
\item 
The covariant ring of a binary quadratic is
the polynomial ring generated by
\begin{align}
f\ \ \mbox{(the form itself)}\quad\mbox{and}\quad
D:=\tv{f}{f}{2}\ \ \mbox{(the discriminant)}.
\end{align}

\item 
The covariant ring of a binary cubic is generated by
three covariants
\begin{align}
g\ \ \mbox{(the form itself)},\quad
P:=\tv{g}{g}{2}\ \ \mbox{(the Hessian)},\quad
Q:=\tv{g}{P}{1}
\end{align}
and one invariant
\begin{align}
\cD:=\tv{P}{P}{2}\ \ \mbox{(the discriminant)}.
\end{align}
They satisfy the unique syzygy
\begin{align}
2Q^2+P^3+g^2\cD=0.
\label{eq:syzygy}
\end{align}
\end{enumerate}
The degree and order of these generators are summarized
in Table~\ref{table:covgen}.

\begin{table}[t]
\begin{align*}
\begin{array}{|c|ccc|}
\hline
d\backslash\omega&0&1&2\\ \hline
 1& -& -& f\\
 2& D=\tv{f}{f}{2}& -& -\\[.2ex]
\hline
\end{array}
\quad
\begin{array}{|c|cccc|}
\hline
d\backslash\omega&0&1&2&3\\ \hline
 1& -& -& -& g\\
 2& -& -&\ P=\tv{g}{g}{2}\ & -\\
 3& -& -& -&\ Q=\tv{g}{P}{1}\ \\
 4& \cD=\tv{P}{P}{2}& -& -& -\\[.2ex]
\hline
\end{array}
\end{align*}
\caption{Generators of the covariant rings
$\bbC[V_2\oplus\bbC^2]^{\grp{SL_2(\bbC)}}$
and
$\bbC[V_3\oplus\bbC^2]^{\grp{SL_2(\bbC)}}$,
arranged by degree $d$ and order $\omega$.}
\label{table:covgen}
\end{table}

We next move on to the joint covariant ring
$\bbC[V_2\oplus V_3\oplus\bbC^2]^{\grp{SL_2(\bbC)}}$.
Basis elements of the ring
are taken from the set
\begin{align}
\left\{\tv{f^{n_1}D^{n_2}}{g^{n_3}P^{n_4}Q^{n_5}\cD^{n_6}}{r}
\mid n_1,n_2,\ldots,n_6,r\in\bbZ_{\ge 0}\right\}.
\end{align}
Clearly, any basis must include
the covariants $f,D,g,P,Q,\cD$.
We also need some additional joint covariants.
To search for these complements,
the following theorem is useful
(see \cite[Section 141]{Grace:1903} for a proof):
\begin{thm}
Let $\{C_1,\ldots,C_l\}$ be a covariant basis of
$\bbC[V_n\oplus\bbC^2]^{\grp{SL_2(\bbC)}}$
and $\{f,D\}$ the covariant basis
of $\bbC[V_2\oplus\bbC^2]^{\grp{SL_2(\bbC)}}$ as above.
Then irreducible joint covariants
in $\bbC[V_2\oplus V_n\oplus\bbC^2]^{\grp{SL_2(\bbC)}}$
can be taken from the set
\begin{enumerate}
\item $\tv{f^r}{C_i}{\rho}$ with $\rho=2r-1,2r$;
\item $\tv{f^r}{C_iC_j}{2r}$ with $C_i$ and $C_j$
are of odd order $k$ and $l$ satisfying $k+l=2r$,
\end{enumerate}
where $r\in\bbZ_{>0}$.
\end{thm}

Note that
in order for the transvectant $\tv{f^r}{C_i}{\rho}$ to be nonzero,
$\rho$ cannot be greater than the order of $C_i$.
Thus the above theorem reduces
the exhaustive search for potential basis elements
to a simple combinatorial enumeration problem.
It is easy to see that
the problem has the following 11 solutions:
\begin{align}
\begin{aligned}
&
\tv{f}{g}{1},
\tv{f}{g}{2},
\tv{f^2}{g}{3},
\tv{f}{P}{1},
\tv{f}{P}{2},
\tv{f}{Q}{1},
\tv{f}{Q}{2},
\tv{f^2}{Q}{3},\\
&
\tv{f^3}{g^2}{6},
\tv{f^3}{gQ}{6},
\tv{f^3}{Q^2}{6}.
\end{aligned}
\end{align}
Among these, two covariants are actually reducible.
One is $\tv{f}{Q}{1}$,
which is rewritten as\footnote{This is due to the following
property of the first transvectants:
if $F,G,H$ are binary forms of order greater than one,
then $\tv{F}{\tv{G}{H}{1}}{1}$ is
reducible \cite[Section 77]{Grace:1903}.}
\begin{align}
\tv{f}{Q}{1}=\tv{f}{\tv{g}{P}{1}}{1}
 =\frac{1}{2}g\cdot\tv{f}{P}{2}-\frac{1}{2}P\cdot\tv{f}{g}{2}.
\end{align}
The other is $\tv{f^3}{Q^2}{6}$,
which is reducible due to the syzygy \eqref{eq:syzygy}.
Indeed, it is expressed as
\begin{align}
\tv{f^3}{Q^2}{6}
 =-\frac{1}{2}\cD\cdot\tv{f^3}{g^2}{6}
  +\frac{3}{10}D\cD\cdot\tv{f}{P}{2}
  -\frac{1}{2}\left\{\tv{f}{P}{2}\right\}^3.
\end{align}

We thus arrive at the following theorem
\cite[Abs\"{a}tze 296--299]{Gordan:1887}
\cite[Section 140]{Grace:1903}:
\begin{thm}\label{thm:15cov}
The covariant ring
$\bbC[V_2\oplus V_3\oplus\bbC^2]^{\grp{SL_2(\bbC)}}$
is generated by the following 15 covariants:
\begin{enumerate}
\item Two covariants of a quadratic: $f,D$;
\item Four covariants of a cubic: $g,P,Q,\cD$;
\item Nine joint covariants of a quadratic and a cubic:
\begin{align*}
\tv{f}{g}{1},
\tv{f}{g}{2},
\tv{f^2}{g}{3},
\tv{f}{P}{1},
\tv{f}{P}{2},
\tv{f}{Q}{2},
\tv{f^2}{Q}{3},
\tv{f^3}{g^2}{6},
\tv{f^3}{gQ}{6}.
\end{align*}
\end{enumerate}
\end{thm}

Figure~\ref{Fig:15cov} shows
graphical expressions for these 15 covariants.
For the sake of conciseness, 
we have employed the macroscopic (or shorthand)
notation of~\cite{Abdesselam:2012}:
all diagrams in Figure~\ref{Fig:15cov}
consist of essentially three types of elements,
which are understood in terms of the original
diagrammatic notation
explained in \eqref{eq:micro1}--\eqref{eq:micro3} as
\begin{align}
\raisebox{-2.5ex}{
\begin{tikzpicture}
\draw[line width=1.2pt](0,0)--(0,.4);
\draw[line width=.8pt,-Stealth](0,.18)--(0,.08);
\draw (.03,.22) node [left]{{\scriptsize$m$}};
\draw[fill=black!10,line width=1pt](0,.7)circle(.3);
\node at (0,.7){$F$};
\end{tikzpicture}
}
=
\raisebox{-4ex}{
\begin{tikzpicture}
\draw[line width=1pt](0,0)--(-.274,-.752);
\draw[line width=1pt](0,0)--(.274,-.752);
\draw[fill=black!10,line width=1pt](0,0)circle(.4);
\node at (0,0){$F$};
\filldraw(-.104,-.591)circle(.01);
\filldraw(0,-.6)circle(.01);
\filldraw(.104,-.591)circle(.01);
\node at (0,-.85){{\scriptsize$m$}};
\end{tikzpicture}
}
,\qquad
\raisebox{-3ex}{
\begin{tikzpicture}
\draw[line width=1.2pt](0,0)--(0,.7);
\draw[line width=.8pt,-Stealth](0,.48)--(0,.38);
\draw (.03,.52) node [left]{{\scriptsize$m$}};
\draw[fill=black!10,line width=1pt](0,0)circle(.3);
\node at (0,0){$u$};
\end{tikzpicture}
}
=
\raisebox{-3ex}{
\begin{tikzpicture}
\draw[line width=1pt](-.3,0)--(-.3,-.6);
\draw[line width=1pt](.3,0)--(.3,-.6);
\draw[fill=black!10,line width=1pt](-.3,-.6)circle(.2);
\node at (-.3,-.6){$u$};
\draw[fill=black!10,line width=1pt](.3,-.6)circle(.2);
\node at (.3,-.6){$u$};
\filldraw(-.15,-.2)circle(.01);
\filldraw(0,-.2)circle(.01);
\filldraw(.15,-.2)circle(.01);
\node at (0,0.1){\scriptsize$m$};
\end{tikzpicture}
}
,\qquad
\raisebox{-3ex}{
\begin{tikzpicture}
\draw[line width=1.2pt](0,.3)--(0,.9);
\draw[line width=1.2pt](0,.9)--(-.4,1.5);
\draw[line width=1.2pt](0,.9)--(.4,1.5);
\draw[line width=.8pt,-Stealth](0,.58)--(0,.48);
\draw[line width=.8pt,-Stealth](-.36,1.44)--(-.16,1.14);
\draw[line width=.8pt,-Stealth](.36,1.44)--(.16,1.14);
\draw (.03,.62) node [left]{{\scriptsize $m$}};
\draw (-.16,1.14) node [left]{{\scriptsize $m_1$}};
\draw (.16,1.14) node [right]{{\scriptsize $m_2$}};
\draw[line width=.6pt](-.15,1.1) to [out=50,in=130] (.15,1.1);
\draw[line width=.2pt,-Stealth](-.12,1.12)--(-.15,1.1);
\filldraw(0,.9)circle(.05);
\end{tikzpicture}
}
=
\raisebox{-7.5ex}{
\begin{tikzpicture}
\draw[line width=1pt](-.3,.2)--(-.3,.6);
\draw[line width=1pt](.3,.2)--(.3,.6);
\draw[line width=1pt](-.9,2.2)--(-.9,2.6);
\draw[line width=1pt](-.3,2.2)--(-.3,2.6);
\draw[line width=1pt](.3,2.2)--(.3,2.6);
\draw[line width=1pt](.9,2.2)--(.9,2.6);
\node at (0,.1){{\scriptsize $m$}};
\node at (-.6,2.7){{\scriptsize $m_1$}};
\node at (.6,2.7){{\scriptsize $m_2$}};
\node at (0,1.82){{\scriptsize $k$}};
\filldraw(-.15,.4)circle(.01);
\filldraw(0,.4)circle(.01);
\filldraw(.15,.4)circle(.01);
\filldraw(-.75,2.4)circle(.01);
\filldraw(-.6,2.4)circle(.01);
\filldraw(.45,2.4)circle(.01);
\filldraw(-.45,2.4)circle(.01);
\filldraw(.6,2.4)circle(.01);
\filldraw(.75,2.4)circle(.01);
\filldraw(-.25,1.08)circle(.01);
\filldraw(-.19,1.08)circle(.01);
\filldraw(-.13,1.08)circle(.01);
\filldraw(.13,1.08)circle(.01);
\filldraw(.19,1.08)circle(.01);
\filldraw(.25,1.08)circle(.01);
\filldraw(-.83,1.72)circle(.01);
\filldraw(-.77,1.72)circle(.01);
\filldraw(-.71,1.72)circle(.01);
\filldraw(-.46,1.72)circle(.01);
\filldraw(-.40,1.72)circle(.01);
\filldraw(-.34,1.72)circle(.01);
\filldraw(.34,1.72)circle(.01);
\filldraw(.40,1.72)circle(.01);
\filldraw(.46,1.72)circle(.01);
\filldraw(.71,1.72)circle(.01);
\filldraw(.77,1.72)circle(.01);
\filldraw(.83,1.72)circle(.01);
\draw[fill=black!10,line width=1pt] (-.4,.6) rectangle (.4,1);
\draw[fill=black!10,line width=1pt] (-1,1.8) rectangle (-.2,2.2);
\draw[fill=black!10,line width=1pt] (.2,1.8) rectangle (1,2.2);
\draw[line width=1pt](-.3,1.8) to [out=270,in=270] (.3,1.8);
\draw[line width=1pt](-.54,1.8) to [out=270,in=270] (.54,1.8);
\draw[line width=1pt](-.66,1.8) to [out=270,in=90] (-.06,1);
\draw[line width=1pt](-.9,1.8) to [out=270,in=90] (-.3,1);
\draw[line width=1pt](.66,1.8) to [out=270,in=90] (.06,1);
\draw[line width=1pt](.9,1.8) to [out=270,in=90] (.3,1);
\draw[line width=.35pt,-Stealth](-.09,1.63)--(-.1,1.63);
\draw[line width=.35pt,-Stealth](-.09,1.49)--(-.1,1.49);
\end{tikzpicture}
}.
\end{align}
In the last equation, $m=m_1+m_2-2k$ and
\begin{align}
\raisebox{-5ex}{
\begin{tikzpicture}
\draw[line width=1pt](-.3,.2)--(-.3,1.4);
\draw[line width=1pt](.3,.2)--(.3,1.4);
\node at (-.3,.05){{\scriptsize$i_1$}};
\node at (.3,.05){{\scriptsize$i_m$}};
\node at (-.3,1.6){{\scriptsize$j_1$}};
\node at (.3,1.6){{\scriptsize$j_m$}};
\filldraw(-.15,.4)circle(.01);
\filldraw(0,.4)circle(.01);
\filldraw(.15,.4)circle(.01);
\filldraw(-.15,1.2)circle(.01);
\filldraw(0,1.2)circle(.01);
\filldraw(.15,1.2)circle(.01);
\draw[fill=black!10,line width=1pt] (-.4,.6) rectangle (.4,1);
\end{tikzpicture}
}
=\ \frac{1}{m!}\sum_{\sigma\in S_m}
\delta_{i_1,j_{\sigma(1)}}\cdots\delta_{i_m,j_{\sigma(m)}}
\end{align}
is the symmetrizer.
Using these diagrammatic expressions,
it is easy to count the degrees and order of covariants:
$d_a$ and $d_b$ are
the number of $f$-nodes and $g$-nodes, respectively;
$\omega$ is the number attached to the arrow going into
the big $u$-node (in the macroscopic notation)
or the number of small $u$-nodes (in the original notation).
The degree and order of covariants are summarized in
Table~\ref{table:jcovgen}.

\begin{figure}[t]
\begin{align*}
\begin{array}{c@{\hspace{1.8em}}c@{\hspace{1.8em}}
c@{\hspace{1.8em}}c@{\hspace{1.8em}}c@{\hspace{1.8em}}
c@{\hspace{1.8em}}c}
\begin{tikzpicture}
\draw[line width=1.2pt](0,0)--(0,.9);
\draw[line width=.8pt,-Stealth](0,.58)--(0,.48);
\draw (.03,.62) node [left]{{\scriptsize 2}};
\draw[fill=black!10,line width=1pt](0,0)circle(.3);
\draw[fill=black!10,line width=1pt](0,1.2)circle(.3);
\node at (0,0){$u$};
\node at (0,1.2){$f$};
\end{tikzpicture}
&
\begin{tikzpicture}
\draw[line width=1.2pt](0,0)--(0,.9);
\draw[line width=.8pt,-Stealth](0,.58)--(0,.48);
\draw (.03,.62) node [left]{{\scriptsize 3}};
\draw[fill=black!10,line width=1pt](0,0)circle(.3);
\draw[fill=black!10,line width=1pt](0,1.2)circle(.3);
\node at (0,0){$u$};
\node at (0,1.2){$g$};
\end{tikzpicture}
&
\begin{tikzpicture}
\draw[line width=1.2pt](0,0)--(0,.9);
\draw[line width=1.2pt](0,.9)--(-.4,1.5);
\draw[line width=1.2pt](0,.9)--(.4,1.5);
\draw[line width=.8pt,-Stealth](0,.58)--(0,.48);
\draw[line width=.8pt,-Stealth](-.36,1.44)--(-.16,1.14);
\draw[line width=.8pt,-Stealth](.36,1.44)--(.16,1.14);
\draw (.03,.62) node [left]{{\scriptsize 0}};
\draw (-.16,1.14) node [left]{{\scriptsize 2}};
\draw (.16,1.14) node [right]{{\scriptsize 2}};
\draw[line width=.6pt](-.15,1.1) to [out=50,in=130] (.15,1.1);
\draw[line width=.2pt,-Stealth](-.12,1.12)--(-.15,1.1);
\filldraw(0,.9)circle(.05);
\draw[fill=black!10,line width=1pt](0,0)circle(.3);
\draw[fill=black!10,line width=1pt](-.4,1.8)circle(.3);
\draw[fill=black!10,line width=1pt](.4,1.8)circle(.3);
\node at (0,0){$u$};
\node at (-.4,1.8){$f$};
\node at (.4,1.8){$f$};
\end{tikzpicture}
&
\begin{tikzpicture}
\draw[line width=1.2pt](0,0)--(0,.9);
\draw[line width=1.2pt](0,.9)--(-.4,1.5);
\draw[line width=1.2pt](0,.9)--(.4,1.5);
\draw[line width=.8pt,-Stealth](0,.58)--(0,.48);
\draw[line width=.8pt,-Stealth](-.36,1.44)--(-.16,1.14);
\draw[line width=.8pt,-Stealth](.36,1.44)--(.16,1.14);
\draw (.03,.62) node [left]{{\scriptsize 3}};
\draw (-.16,1.14) node [left]{{\scriptsize 2}};
\draw (.16,1.14) node [right]{{\scriptsize 3}};
\draw[line width=.6pt](-.15,1.1) to [out=50,in=130] (.15,1.1);
\draw[line width=.2pt,-Stealth](-.12,1.12)--(-.15,1.1);
\filldraw(0,.9)circle(.05);
\draw[fill=black!10,line width=1pt](0,0)circle(.3);
\draw[fill=black!10,line width=1pt](-.4,1.8)circle(.3);
\draw[fill=black!10,line width=1pt](.4,1.8)circle(.3);
\node at (0,0){$u$};
\node at (-.4,1.8){$f$};
\node at (.4,1.8){$g$};
\end{tikzpicture}
&
\begin{tikzpicture}
\draw[line width=1.2pt](0,0)--(0,.9);
\draw[line width=1.2pt](0,.9)--(-.4,1.5);
\draw[line width=1.2pt](0,.9)--(.4,1.5);
\draw[line width=.8pt,-Stealth](0,.58)--(0,.48);
\draw[line width=.8pt,-Stealth](-.36,1.44)--(-.16,1.14);
\draw[line width=.8pt,-Stealth](.36,1.44)--(.16,1.14);
\draw (.03,.62) node [left]{{\scriptsize 1}};
\draw (-.16,1.14) node [left]{{\scriptsize 2}};
\draw (.16,1.14) node [right]{{\scriptsize 3}};
\draw[line width=.6pt](-.15,1.1) to [out=50,in=130] (.15,1.1);
\draw[line width=.2pt,-Stealth](-.12,1.12)--(-.15,1.1);
\filldraw(0,.9)circle(.05);
\draw[fill=black!10,line width=1pt](0,0)circle(.3);
\draw[fill=black!10,line width=1pt](-.4,1.8)circle(.3);
\draw[fill=black!10,line width=1pt](.4,1.8)circle(.3);
\node at (0,0){$u$};
\node at (-.4,1.8){$f$};
\node at (.4,1.8){$g$};
\end{tikzpicture}
&
\begin{tikzpicture}
\draw[line width=1.2pt](0,0)--(0,.9);
\draw[line width=1.2pt](0,.9)--(-.4,1.5);
\draw[line width=1.2pt](0,.9)--(.4,1.5);
\draw[line width=.8pt,-Stealth](0,.58)--(0,.48);
\draw[line width=.8pt,-Stealth](-.36,1.44)--(-.16,1.14);
\draw[line width=.8pt,-Stealth](.36,1.44)--(.16,1.14);
\draw (.03,.62) node [left]{{\scriptsize 2}};
\draw (-.16,1.14) node [left]{{\scriptsize 3}};
\draw (.16,1.14) node [right]{{\scriptsize 3}};
\draw[line width=.6pt](-.15,1.1) to [out=50,in=130] (.15,1.1);
\draw[line width=.2pt,-Stealth](-.12,1.12)--(-.15,1.1);
\filldraw(0,.9)circle(.05);
\draw[fill=black!10,line width=1pt](0,0)circle(.3);
\draw[fill=black!10,line width=1pt](-.4,1.8)circle(.3);
\draw[fill=black!10,line width=1pt](.4,1.8)circle(.3);
\node at (0,0){$u$};
\node at (-.4,1.8){$g$};
\node at (.4,1.8){$g$};
\end{tikzpicture}
&
\begin{tikzpicture}
\draw[line width=1.2pt](0,0)--(0,.9);
\draw[line width=1.2pt](0,.9)--(-.8,2.1);
\draw[line width=1.2pt](0,.9)--(.8,2.1);
\draw[line width=1.2pt](.4,1.5)--(0,2.1);
\draw[line width=.8pt,-Stealth](0,.58)--(0,.48);
\draw[line width=.8pt,-Stealth](-.56,1.74)--(-.36,1.44);
\draw[line width=.8pt,-Stealth](.04,2.04)--(.24,1.74);
\draw[line width=.8pt,-Stealth](.76,2.04)--(.56,1.74);
\draw[line width=.8pt,-Stealth](.36,1.44)--(.16,1.14);
\draw (-.03,.62) node [right]{{\scriptsize 3}};
\draw (-.36,1.44) node [left]{{\scriptsize 3}};
\draw (.24,1.74) node [left]{{\scriptsize 3}};
\draw (.56,1.74) node [right]{{\scriptsize 3}};
\draw (.16,1.14) node [right]{{\scriptsize 2}};
\draw[line width=.6pt](-.15,1.1) to [out=50,in=130] (.15,1.1);
\draw[line width=.2pt,-Stealth](-.12,1.12)--(-.15,1.1);
\draw[line width=.6pt](.25,1.7) to [out=50,in=130] (.55,1.7);
\draw[line width=.2pt,-Stealth](.28,1.72)--(.25,1.7);
\filldraw(0,.9)circle(.05);
\filldraw(.4,1.5)circle(.05);
\draw[fill=black!10,line width=1pt](0,0)circle(.3);
\draw[fill=black!10,line width=1pt](-.8,2.4)circle(.3);
\draw[fill=black!10,line width=1pt](0,2.4)circle(.3);
\draw[fill=black!10,line width=1pt](.8,2.4)circle(.3);
\node at (0,0){$u$};
\node at (-.8,2.4){$g$};
\node at (0,2.4){$g$};
\node at (.8,2.4){$g$};
\end{tikzpicture}
\\[1ex]
f&g&\hspace{-1em}\ D=\tv{f}{f}{2}\hspace{-1em}
&\ \tv{f}{g}{1}&\ \tv{f}{g}{2}
&\hspace{-1em}\ P=\tv{g}{g}{2}\hspace{-1em}
&\ Q=\tv{g}{P}{1}
\end{array}
\end{align*}
\begin{align*}
\begin{array}{c@{\hspace{1.2em}}c@{\hspace{1.2em}}c
@{\hspace{1.2em}}c@{\hspace{1.2em}}c}
\begin{tikzpicture}
\draw[line width=1.2pt](0,0)--(0,.9);
\draw[line width=1.2pt](0,.9)--(-.8,2.1);
\draw[line width=1.2pt](0,.9)--(.8,2.1);
\draw[line width=1.2pt](-.4,1.5)--(0,2.1);
\draw[line width=.8pt,-Stealth](0,.58)--(0,.48);
\draw[line width=.8pt,-Stealth](-.76,2.04)--(-.56,1.74);
\draw[line width=.8pt,-Stealth](-.36,1.44)--(-.16,1.14);
\draw[line width=.8pt,-Stealth](-.04,2.04)--(-.24,1.74);
\draw[line width=.8pt,-Stealth](.56,1.74)--(.36,1.44);
\draw (.03,.62) node [left]{{\scriptsize 1}};
\draw (-.56,1.74) node [left]{{\scriptsize 2}};
\draw (-.16,1.14) node [left]{{\scriptsize 4}};
\draw (-.24,1.74) node [right]{{\scriptsize 2}};
\draw (.36,1.44) node [right]{{\scriptsize 3}};
\draw[line width=.6pt](-.15,1.1) to [out=50,in=130] (.15,1.1);
\draw[line width=.2pt,-Stealth](-.12,1.12)--(-.15,1.1);
\draw[line width=.6pt](-.55,1.7) to [out=50,in=130] (-.25,1.7);
\draw[line width=.2pt,-Stealth](-.52,1.72)--(-.55,1.7);
\filldraw(0,.9)circle(.05);
\filldraw(-.4,1.5)circle(.05);
\draw[fill=black!10,line width=1pt](0,0)circle(.3);
\draw[fill=black!10,line width=1pt](-.8,2.4)circle(.3);
\draw[fill=black!10,line width=1pt](0,2.4)circle(.3);
\draw[fill=black!10,line width=1pt](.8,2.4)circle(.3);
\node at (0,0){$u$};
\node at (-.8,2.4){$f$};
\node at (0,2.4){$f$};
\node at (.8,2.4){$g$};
\end{tikzpicture}
&
\begin{tikzpicture}
\draw[line width=1.2pt](0,0)--(0,.9);
\draw[line width=1.2pt](0,.9)--(-.8,2.1);
\draw[line width=1.2pt](0,.9)--(.8,2.1);
\draw[line width=1.2pt](.4,1.5)--(0,2.1);
\draw[line width=.8pt,-Stealth](0,.58)--(0,.48);
\draw[line width=.8pt,-Stealth](-.56,1.74)--(-.36,1.44);
\draw[line width=.8pt,-Stealth](.04,2.04)--(.24,1.74);
\draw[line width=.8pt,-Stealth](.76,2.04)--(.56,1.74);
\draw[line width=.8pt,-Stealth](.36,1.44)--(.16,1.14);
\draw (-.03,.62) node [right]{{\scriptsize 2}};
\draw (-.36,1.44) node [left]{{\scriptsize 2}};
\draw (.24,1.74) node [left]{{\scriptsize 3}};
\draw (.56,1.74) node [right]{{\scriptsize 3}};
\draw (.16,1.14) node [right]{{\scriptsize 2}};
\draw[line width=.6pt](-.15,1.1) to [out=50,in=130] (.15,1.1);
\draw[line width=.2pt,-Stealth](-.12,1.12)--(-.15,1.1);
\draw[line width=.6pt](.25,1.7) to [out=50,in=130] (.55,1.7);
\draw[line width=.2pt,-Stealth](.28,1.72)--(.25,1.7);
\filldraw(0,.9)circle(.05);
\filldraw(.4,1.5)circle(.05);
\draw[fill=black!10,line width=1pt](0,0)circle(.3);
\draw[fill=black!10,line width=1pt](-.8,2.4)circle(.3);
\draw[fill=black!10,line width=1pt](0,2.4)circle(.3);
\draw[fill=black!10,line width=1pt](.8,2.4)circle(.3);
\node at (0,0){$u$};
\node at (-.8,2.4){$f$};
\node at (0,2.4){$g$};
\node at (.8,2.4){$g$};
\end{tikzpicture}
&
\begin{tikzpicture}
\draw[line width=1.2pt](0,0)--(0,.9);
\draw[line width=1.2pt](0,.9)--(-.8,2.1);
\draw[line width=1.2pt](0,.9)--(.8,2.1);
\draw[line width=1.2pt](.4,1.5)--(0,2.1);
\draw[line width=.8pt,-Stealth](0,.58)--(0,.48);
\draw[line width=.8pt,-Stealth](-.56,1.74)--(-.36,1.44);
\draw[line width=.8pt,-Stealth](.04,2.04)--(.24,1.74);
\draw[line width=.8pt,-Stealth](.76,2.04)--(.56,1.74);
\draw[line width=.8pt,-Stealth](.36,1.44)--(.16,1.14);
\draw (-.03,.62) node [right]{{\scriptsize 0}};
\draw (-.36,1.44) node [left]{{\scriptsize 2}};
\draw (.24,1.74) node [left]{{\scriptsize 3}};
\draw (.56,1.74) node [right]{{\scriptsize 3}};
\draw (.16,1.14) node [right]{{\scriptsize 2}};
\draw[line width=.6pt](-.15,1.1) to [out=50,in=130] (.15,1.1);
\draw[line width=.2pt,-Stealth](-.12,1.12)--(-.15,1.1);
\draw[line width=.6pt](.25,1.7) to [out=50,in=130] (.55,1.7);
\draw[line width=.2pt,-Stealth](.28,1.72)--(.25,1.7);
\filldraw(0,.9)circle(.05);
\filldraw(.4,1.5)circle(.05);
\draw[fill=black!10,line width=1pt](0,0)circle(.3);
\draw[fill=black!10,line width=1pt](-.8,2.4)circle(.3);
\draw[fill=black!10,line width=1pt](0,2.4)circle(.3);
\draw[fill=black!10,line width=1pt](.8,2.4)circle(.3);
\node at (0,0){$u$};
\node at (-.8,2.4){$f$};
\node at (0,2.4){$g$};
\node at (.8,2.4){$g$};
\end{tikzpicture}
&
\begin{tikzpicture}
\draw[line width=1.2pt](0,0)--(0,.9);
\draw[line width=1.2pt](0,.9)--(-1.2,2.7);
\draw[line width=1.2pt](0,.9)--(1.2,2.7);
\draw[line width=1.2pt](.4,1.5)--(-.4,2.7);
\draw[line width=1.2pt](.8,2.1)--(.4,2.7);
\draw[line width=.8pt,-Stealth](0,.58)--(0,.48);
\draw[line width=.8pt,-Stealth](-.76,2.04)--(-.56,1.74);
\draw[line width=.8pt,-Stealth](-.16,2.34)--(.04,2.04);
\draw[line width=.8pt,-Stealth](.44,2.64)--(.64,2.34);
\draw[line width=.8pt,-Stealth](1.16,2.64)--(.96,2.34);
\draw[line width=.8pt,-Stealth](.76,2.04)--(.56,1.74);
\draw[line width=.8pt,-Stealth](.36,1.44)--(.16,1.14);
\draw (-.03,.62) node [right]{{\scriptsize 1}};
\draw (-.56,1.74) node [left]{{\scriptsize 2}};
\draw (.04,2.04) node [left]{{\scriptsize 3}};
\draw (.64,2.34) node [left]{{\scriptsize 3}};
\draw (.96,2.34) node [right]{{\scriptsize 3}};
\draw (.56,1.74) node [right]{{\scriptsize 2}};
\draw (.16,1.14) node [right]{{\scriptsize 3}};
\draw[line width=.6pt](-.15,1.1) to [out=50,in=130] (.15,1.1);
\draw[line width=.2pt,-Stealth](-.12,1.12)--(-.15,1.1);
\draw[line width=.6pt](.25,1.7) to [out=50,in=130] (.55,1.7);
\draw[line width=.2pt,-Stealth](.28,1.72)--(.25,1.7);
\draw[line width=.6pt](.65,2.3) to [out=50,in=130] (.95,2.3);
\draw[line width=.2pt,-Stealth](.68,2.32)--(.65,2.3);
\filldraw(0,.9)circle(.05);
\filldraw(.4,1.5)circle(.05);
\filldraw(.8,2.1)circle(.05);
\draw[fill=black!10,line width=1pt](0,0)circle(.3);
\draw[fill=black!10,line width=1pt](-1.2,3)circle(.3);
\draw[fill=black!10,line width=1pt](-.4,3)circle(.3);
\draw[fill=black!10,line width=1pt](.4,3)circle(.3);
\draw[fill=black!10,line width=1pt](1.2,3)circle(.3);
\node at (0,0){$u$};
\node at (-1.2,3){$f$};
\node at (-.4,3){$g$};
\node at (.4,3){$g$};
\node at (1.2,3){$g$};
\end{tikzpicture}
&
\begin{tikzpicture}
\draw[line width=1.2pt](0,0)--(0,.9);
\draw[line width=1.2pt](0,.9)--(-1.2,2.7);
\draw[line width=1.2pt](0,.9)--(1.2,2.7);
\draw[line width=1.2pt](-.8,2.1)--(-.4,2.7);
\draw[line width=1.2pt](.8,2.1)--(.4,2.7);
\draw[line width=.8pt,-Stealth](0,.58)--(0,.48);
\draw[line width=.8pt,-Stealth](-1.16,2.64)--(-.96,2.34);
\draw[line width=.8pt,-Stealth](-.44,2.64)--(-.64,2.34);
\draw[line width=.8pt,-Stealth](-.56,1.74)--(-.36,1.44);
\draw[line width=.8pt,-Stealth](.44,2.64)--(.64,2.34);
\draw[line width=.8pt,-Stealth](1.16,2.64)--(.96,2.34);
\draw[line width=.8pt,-Stealth](.56,1.74)--(.36,1.44);
\draw (-.03,.62) node [right]{{\scriptsize 0}};
\draw (-.96,2.34) node [left]{{\scriptsize 3}};
\draw (-.64,2.34) node [right]{{\scriptsize 3}};
\draw (-.36,1.44) node [left]{{\scriptsize 2}};
\draw (.64,2.34) node [left]{{\scriptsize 3}};
\draw (.96,2.34) node [right]{{\scriptsize 3}};
\draw (.36,1.44) node [right]{{\scriptsize 2}};
\draw[line width=.6pt](-.15,1.1) to [out=50,in=130] (.15,1.1);
\draw[line width=.2pt,-Stealth](-.12,1.12)--(-.15,1.1);
\draw[line width=.6pt](-.95,2.3) to [out=50,in=130] (-.65,2.3);
\draw[line width=.2pt,-Stealth](-.92,2.32)--(-.95,2.3);
\draw[line width=.6pt](.65,2.3) to [out=50,in=130] (.95,2.3);
\draw[line width=.2pt,-Stealth](.68,2.32)--(.65,2.3);
\filldraw(0,.9)circle(.05);
\filldraw(-.8,2.1)circle(.05);
\filldraw(.8,2.1)circle(.05);
\draw[fill=black!10,line width=1pt](0,0)circle(.3);
\draw[fill=black!10,line width=1pt](-1.2,3)circle(.3);
\draw[fill=black!10,line width=1pt](-.4,3)circle(.3);
\draw[fill=black!10,line width=1pt](.4,3)circle(.3);
\draw[fill=black!10,line width=1pt](1.2,3)circle(.3);
\node at (0,0){$u$};
\node at (-1.2,3){$g$};
\node at (-.4,3){$g$};
\node at (.4,3){$g$};
\node at (1.2,3){$g$};
\end{tikzpicture}
\\[1ex]
\ \tv{f^2}{g}{3}&\ \tv{f}{P}{1}&\ \tv{f}{P}{2}
&\ \tv{f}{Q}{2}&\ \cD=\tv{P}{P}{2}
\end{array}
\end{align*}
\begin{align*}
\begin{array}{c@{\hspace{1.8em}}c@{\hspace{1.8em}}c}
\begin{tikzpicture}
\draw[line width=1.2pt](0,0)--(0,.9);
\draw[line width=1.2pt](0,.9)--(-1.6,3.3);
\draw[line width=1.2pt](-1.2,2.7)--(-.8,3.3);
\draw[line width=1.2pt](-.8,2.1)--(0,3.3);
\draw[line width=1.2pt](1.2,2.7)--(.8,3.3);
\draw[line width=1.2pt](0,.9)--(1.6,3.3);
\draw[line width=.8pt,-Stealth](0,.58)--(0,.48);
\draw[line width=.8pt,-Stealth](-1.56,3.24)--(-1.36,2.94);
\draw[line width=.8pt,-Stealth](-.56,1.74)--(-.36,1.44);
\draw[line width=.8pt,-Stealth](-.84,3.24)--(-1.04,2.94);
\draw[line width=.8pt,-Stealth](-.24,2.94)--(-.44,2.64);
\draw[line width=.8pt,-Stealth](.84,3.24)--(1.04,2.94);
\draw[line width=.8pt,-Stealth](1.56,3.24)--(1.36,2.94);
\draw[line width=.8pt,-Stealth](-1.16,2.64)--(-.96,2.34);
\draw[line width=.8pt,-Stealth](0.76,2.04)--(0.56,1.74);
\draw (-.03,.62) node [right]{{\scriptsize 0}};
\draw (-1.36,2.94) node [left]{{\scriptsize 2}};
\draw (-.36,1.44) node [left]{{\scriptsize 6}};
\draw (-1.04,2.94) node [right]{{\scriptsize 2}};
\draw (-.44,2.64) node [right]{{\scriptsize 2}};
\draw (1.04,2.94) node [left]{{\scriptsize 3}};
\draw (1.36,2.94) node [right]{{\scriptsize 3}};
\draw (-.96,2.34) node [left]{{\scriptsize 4}};
\draw (0.56,1.74) node [right]{{\scriptsize 6}};
\draw[line width=.6pt](-.15,1.1) to [out=50,in=130] (.15,1.1);
\draw[line width=.2pt,-Stealth](-.12,1.12)--(-.15,1.1);
\draw[line width=.6pt](-1.35,2.9) to [out=50,in=130] (-1.05,2.9);
\draw[line width=.2pt,-Stealth](-1.32,2.92)--(-1.35,2.9);
\draw[line width=.6pt](-.95,2.3) to [out=50,in=130] (-.65,2.3);
\draw[line width=.2pt,-Stealth](-.92,2.32)--(-.95,2.3);
\draw[line width=.6pt](1.05,2.9) to [out=50,in=130] (1.35,2.9);
\draw[line width=.2pt,-Stealth](1.08,2.92)--(1.05,2.9);
\filldraw(0,.9)circle(.05);
\filldraw(-1.2,2.7)circle(.05);
\filldraw(-.8,2.1)circle(.05);
\filldraw(1.2,2.7)circle(.05);
\draw[fill=black!10,line width=1pt](0,0)circle(.3);
\draw[fill=black!10,line width=1pt](-1.6,3.6)circle(.3);
\draw[fill=black!10,line width=1pt](-0.8,3.6)circle(.3);
\draw[fill=black!10,line width=1pt](   0,3.6)circle(.3);
\draw[fill=black!10,line width=1pt]( 0.8,3.6)circle(.3);
\draw[fill=black!10,line width=1pt]( 1.6,3.6)circle(.3);
\node at (0,0){$u$};
\node at (-1.6,3.6){$f$};
\node at (-0.8,3.6){$f$};
\node at (   0,3.6){$f$};
\node at ( 0.8,3.6){$g$};
\node at ( 1.6,3.6){$g$};
\end{tikzpicture}
&
\begin{tikzpicture}
\draw[line width=1.2pt](0,0)--(0,.9);
\draw[line width=1.2pt](0,.9)--(-1.6,3.3);
\draw[line width=1.2pt](-1.2,2.7)--(-.8,3.3);
\draw[line width=1.2pt](.8,2.1)--(0,3.3);
\draw[line width=1.2pt](1.2,2.7)--(.8,3.3);
\draw[line width=1.2pt](0,.9)--(1.6,3.3);
\draw[line width=.8pt,-Stealth](0,.58)--(0,.48);
\draw[line width=.8pt,-Stealth](-1.56,3.24)--(-1.36,2.94);
\draw[line width=.8pt,-Stealth](-.76,2.04)--(-.56,1.74);
\draw[line width=.8pt,-Stealth](-.84,3.24)--(-1.04,2.94);
\draw[line width=.8pt,-Stealth](.24,2.94)--(.44,2.64);
\draw[line width=.8pt,-Stealth](.84,3.24)--(1.04,2.94);
\draw[line width=.8pt,-Stealth](1.56,3.24)--(1.36,2.94);
\draw[line width=.8pt,-Stealth](1.16,2.64)--(.96,2.34);
\draw[line width=.8pt,-Stealth](.56,1.74)--(.36,1.44);
\draw (-.03,.62) node [right]{{\scriptsize 1}};
\draw (-1.36,2.94) node [left]{{\scriptsize 2}};
\draw (-.56,1.74) node [left]{{\scriptsize 4}};
\draw (-1.04,2.94) node [right]{{\scriptsize 2}};
\draw (.44,2.64) node [left]{{\scriptsize 3}};
\draw (1.04,2.94) node [left]{{\scriptsize 3}};
\draw (1.36,2.94) node [right]{{\scriptsize 3}};
\draw (.96,2.34) node [right]{{\scriptsize 2}};
\draw (.36,1.44) node [right]{{\scriptsize 3}};
\draw[line width=.6pt](-.15,1.1) to [out=50,in=130] (.15,1.1);
\draw[line width=.2pt,-Stealth](-.12,1.12)--(-.15,1.1);
\draw[line width=.6pt](-1.35,2.9) to [out=50,in=130] (-1.05,2.9);
\draw[line width=.2pt,-Stealth](-1.32,2.92)--(-1.35,2.9);
\draw[line width=.6pt](.65,2.3) to [out=50,in=130] (.95,2.3);
\draw[line width=.2pt,-Stealth](.68,2.32)--(.65,2.3);
\draw[line width=.6pt](1.05,2.9) to [out=50,in=130] (1.35,2.9);
\draw[line width=.2pt,-Stealth](1.08,2.92)--(1.05,2.9);
\filldraw(0,.9)circle(.05);
\filldraw(-1.2,2.7)circle(.05);
\filldraw(.8,2.1)circle(.05);
\filldraw(1.2,2.7)circle(.05);
\draw[fill=black!10,line width=1pt](0,0)circle(.3);
\draw[fill=black!10,line width=1pt](-1.6,3.6)circle(.3);
\draw[fill=black!10,line width=1pt](-0.8,3.6)circle(.3);
\draw[fill=black!10,line width=1pt](   0,3.6)circle(.3);
\draw[fill=black!10,line width=1pt]( 0.8,3.6)circle(.3);
\draw[fill=black!10,line width=1pt]( 1.6,3.6)circle(.3);
\node at (0,0){$u$};
\node at (-1.6,3.6){$f$};
\node at (-0.8,3.6){$f$};
\node at (   0,3.6){$g$};
\node at ( 0.8,3.6){$g$};
\node at ( 1.6,3.6){$g$};
\end{tikzpicture}
&
\begin{tikzpicture}
\draw[line width=1.2pt](0,0)--(0,.9);
\draw[line width=1.2pt](0,.9)--(-2.4,4.5);
\draw[line width=1.2pt](-2.0,3.9)--(-1.6,4.5);
\draw[line width=1.2pt](-1.6,3.3)--(-.8,4.5);
\draw[line width=1.2pt](1.2,2.7)--(0,4.5);
\draw[line width=1.2pt](1.6,3.3)--(.8,4.5);
\draw[line width=1.2pt](2.0,3.9)--(1.6,4.5);
\draw[line width=1.2pt](0,.9)--(2.4,4.5);
\draw[line width=.8pt,-Stealth](0,.58)--(0,.48);
\draw[line width=.8pt,-Stealth](-2.36,4.44)--(-2.16,4.14);
\draw[line width=.8pt,-Stealth](-1.96,3.84)--(-1.76,3.54);
\draw[line width=.8pt,-Stealth](-0.96,2.34)--(-0.76,2.04);
\draw[line width=.8pt,-Stealth](-1.64,4.44)--(-1.84,4.14);
\draw[line width=.8pt,-Stealth](-1.04,4.14)--(-1.24,3.84);
\draw[line width=.8pt,-Stealth](0.44,3.84)--(0.64,3.54);
\draw[line width=.8pt,-Stealth](1.04,4.14)--(1.24,3.84);
\draw[line width=.8pt,-Stealth](1.64,4.44)--(1.84,4.14);
\draw[line width=.8pt,-Stealth](2.36,4.44)--(2.16,4.14);
\draw[line width=.8pt,-Stealth](1.96,3.84)--(1.76,3.54);
\draw[line width=.8pt,-Stealth](1.56,3.24)--(1.36,2.94);
\draw[line width=.8pt,-Stealth](0.76,2.04)--(0.56,1.74);
\draw (-.03,.62) node [right]{{\scriptsize 0}};
\draw (-2.16,4.14) node [left]{{\scriptsize 2}};
\draw (-1.76,3.54) node [left]{{\scriptsize 4}};
\draw (-.76,2.04) node [left]{{\scriptsize 6}};
\draw (-1.84,4.14) node [right]{{\scriptsize 2}};
\draw (-1.24,3.84) node [right]{{\scriptsize 2}};
\draw (0.64,3.54) node [left]{{\scriptsize 3}};
\draw (1.24,3.84) node [left]{{\scriptsize 3}};
\draw (1.84,4.14) node [left]{{\scriptsize 3}};
\draw (2.16,4.14) node [right]{{\scriptsize 3}};
\draw (1.76,3.54) node [right]{{\scriptsize 2}};
\draw (1.36,2.94) node [right]{{\scriptsize 3}};
\draw (0.56,1.74) node [right]{{\scriptsize 6}};
\draw[line width=.6pt](-.15,1.1) to [out=50,in=130] (.15,1.1);
\draw[line width=.2pt,-Stealth](-.12,1.12)--(-.15,1.1);
\draw[line width=.6pt](-2.15,4.1) to [out=50,in=130] (-1.85,4.1);
\draw[line width=.2pt,-Stealth](-2.12,4.12)--(-2.15,4.1);
\draw[line width=.6pt](-1.75,3.5) to [out=50,in=130] (-1.45,3.5);
\draw[line width=.2pt,-Stealth](-1.72,3.52)--(-1.75,3.5);
\draw[line width=.6pt](1.05,2.9) to [out=50,in=130] (1.35,2.9);
\draw[line width=.2pt,-Stealth](1.08,2.92)--(1.05,2.9);
\draw[line width=.6pt](1.45,3.5) to [out=50,in=130] (1.75,3.5);
\draw[line width=.2pt,-Stealth](1.48,3.52)--(1.45,3.5);
\draw[line width=.6pt](1.85,4.1) to [out=50,in=130] (2.15,4.1);
\draw[line width=.2pt,-Stealth](1.88,4.12)--(1.85,4.1);
\filldraw(0,.9)circle(.05);
\filldraw(-2.0,3.9)circle(.05);
\filldraw(-1.6,3.3)circle(.05);
\filldraw(1.2,2.7)circle(.05);
\filldraw(1.6,3.3)circle(.05);
\filldraw(2.0,3.9)circle(.05);
\draw[fill=black!10,line width=1pt](0,0)circle(.3);
\draw[fill=black!10,line width=1pt](-2.4,4.8)circle(.3);
\draw[fill=black!10,line width=1pt](-1.6,4.8)circle(.3);
\draw[fill=black!10,line width=1pt](-0.8,4.8)circle(.3);
\draw[fill=black!10,line width=1pt](   0,4.8)circle(.3);
\draw[fill=black!10,line width=1pt]( 0.8,4.8)circle(.3);
\draw[fill=black!10,line width=1pt]( 1.6,4.8)circle(.3);
\draw[fill=black!10,line width=1pt]( 2.4,4.8)circle(.3);
\node at (0,0){$u$};
\node at (-2.4,4.8){$f$};
\node at (-1.6,4.8){$f$};
\node at (-0.8,4.8){$f$};
\node at (   0,4.8){$g$};
\node at ( 0.8,4.8){$g$};
\node at ( 1.6,4.8){$g$};
\node at ( 2.4,4.8){$g$};
\end{tikzpicture}
\\[1ex]
\ \tv{f^3}{g^2}{6}&\ \tv{f^2}{Q}{3}&\ \tv{f^3}{gQ}{6}
\end{array}
\end{align*}
\caption{Graphical expressions for 15 covariants.
\label{Fig:15cov}}
\end{figure}
\clearpage

\begin{table}[t]
\begin{align*}
\begin{array}{|c|cccc|c|}
\hline
d_t\backslash\omega&0&1&2&3&\#\\ \hline
 1& -& -& -& -& 0\\
 2& D=\tv{f}{f}{2}& -& -& -& 1\\
 3& \tv{f}{P}{2}& \tv{f}{g}{2}& f& -& 3\\
 4& \cD=\tv{P}{P}{2}& \tv{f^2}{g}{3}&P=\tv{g}{g}{2}& g& 4\\
 5& \tv{f^3}{g^2}{6}& \tv{f}{Q}{2}& \tv{f}{P}{1}& \tv{f}{g}{1}& 4\\
 6& -& \tv{f^2}{Q}{3}& -& Q=\tv{g}{P}{1}& 2\\
 7& \tv{f^3}{gQ}{6}& -& -& -& 1\\ \hline
\mbox{Tot.}&5&4&3&3&15\\
\hline
\end{array}
\end{align*}
\caption{Generators of
$\bbC[V_2\oplus V_3\oplus\bbC^2]^{\grp{SL_2(\bbC)}}$,
arranged by total degree $d_t=d_a+d_b+\omega$ and order $\omega$.}
\label{table:jcovgen}
\end{table}
\begin{thm}
The 15 covariants of Theorem~\ref{thm:15cov} form a minimal basis.
\end{thm}
\begin{proof}
Let $\Rcov_{d_a,d_b,\omega}$ denote the module of
covariants of degrees $(d_a,d_b)$ and order $\omega$.
By observing that
\begin{align}
\bigoplus_{d_a,d_b,\omega}\Rcov_{d_a,d_b,\omega}
=\bbC[V_2\oplus V_3\oplus\bbC^2]^{\grp{SL_2(\bbC)}}
\cong \bbC[V_2\oplus V_3\oplus V_1]^{\grp{SL_2(\bbC)}},
\end{align}
the Hilbert--Poincar\'e series is easily calculated as
(see e.g.~\cite[Section~4.4]{Mukai:2003})
\begin{align}
\begin{aligned}
H(t)
&=\sum_{d_a,d_b,\omega}\mathrm{dim}(\Rcov_{d_a,d_b,\omega})
  t^{d_a+d_b+\omega}\\
&=\mathrm{Res}_{q=0}
\frac{q^{-1}-q}
     {\prod_{i=0}^2(1-tq^{2-2i})\prod_{i=0}^3(1-tq^{3-2i})
      \prod_{i=0}^1(1-tq^{1-2i})}\\[1ex]
&=\frac{1+t^2+t^3+2t^4+4t^5+5t^6+6t^7
        +5t^8+4t^9+2t^{10}+t^{11}+t^{12}+t^{14}}
       {(1-t^3)^2(1-t^4)^3(1-t^5)}\\[1ex]
&=1+t^2+3t^3+5t^4+7t^5+13t^6+20t^7+O(t^8).
\end{aligned}
\label{eq:HilbSer}
\end{align}
On the other hand, taking account of all 15 generators
(see Table~\ref{table:jcovgen} for their total degree
$d_t=d_a+d_b+\omega$),
we have
\begin{align}
\begin{aligned}
&\frac{1}{(1-t^2)(1-t^3)^3(1-t^4)^4(1-t^5)^4(1-t^6)^2(1-t^7)}\\[1ex]
&=1+t^2+3t^3+5t^4+7t^5+13t^6+20t^7+O(t^8).
\end{aligned}
\end{align}
This is in perfect agreement with
the series expansion of $H(t)$ in \eqref{eq:HilbSer}
up to the order of $t^8$.
Hence, there are no redundant basis elements.
\end{proof}
%

\subsection{Minimal basis for the ring of $D_4$ triality invariants}

In this subsection we study a minimal basis of
the ring $\Rtri_{*,*}$ of $D_4$ triality invariants.

Recall that the isomorphism constructed in section~\ref{subsec:iso}
is extremely simple:
for any joint covariant expressed in terms of
\begin{align}
f=\sum_{i=0}^2\alpha_iu^{2-i}v^i,\qquad
g=\sum_{i=0}^3\beta_iu^{3-i}v^i,
\end{align}
the corresponding $D_4$ triality invariant is obtained
by simply setting
\begin{align}
u=1,\quad v=0,\quad
\alpha_i=a_i,\quad
\beta_j=b_j,\quad
a_1=0.
\label{eq:isosubs}
\end{align}
Here $a_i,b_j$ are the meromorphic $D_4$ triality invariants
given in \eqref{eq:abform}. By abuse of notation,
in what follows we write
the basis elements of $\Rtri_{*,*}$ as transvectants
of $f$ and $g$, with the understanding that
the above substitution is performed.

Thus we immediately obtain 15 $D_4$ triality invariants
as the basis elements of $\Rtri_{*,*}$.
Note, however, that the correspondence of the grades of covariants
and $D_4$ triality invariants is not entirely trivial,
as summarized in Remark~\ref{rem:graderel}.
For the study of the ring $\Rtri_{*,*}$ of $D_4$ triality invariants,
it is useful to arrange the basis elements by polynomial degree $m$
and weight $k$. For this purpose, we label the basis elements as
\begin{align}
\gen{d_a}{d_b}{m}{\omega}.
\end{align}
While the second subscript $\omega=2d_a+3d_b-m$ is redundant,
we put it so that the weight of the generators
is easily calculated as $k=3m+2\omega$.

The basis elements of $\Rtri_{*,*}$ are then expressed as follows:
\begin{align}
\gen{1}{0}{0}{2} &= f,&
\gen{0}{1}{0}{3} &= g,&
\gen{1}{1}{2}{3} &= \tv{f}{g}{1},\nn\\
\gen{2}{0}{4}{0} &= \tv{f}{f}{2}=D,&
\gen{1}{1}{4}{1} &= \tv{f}{g}{2},&
\gen{0}{2}{4}{2} &= \tv{g}{g}{2}=P,\nn\\
\gen{2}{1}{6}{1} &= \tv{f^2}{g}{3},&
\gen{1}{2}{6}{2} &= \tv{f}{P}{1},&
\gen{0}{3}{6}{3} &= \tv{g}{P}{1}=Q,\nn\\
\gen{1}{2}{8}{0} &= \tv{f}{P}{2},&
\gen{1}{3}{10}{1} &= \tv{f}{Q}{2},&
\gen{3}{2}{12}{0} &= \tv{f^3}{g^2}{6},\nn\\
\gen{0}{4}{12}{0} &= \tv{P}{P}{2}=\cD,&
\gen{2}{3}{12}{1} &= \tv{f^2}{Q}{3},&
\gen{3}{4}{18}{0} &= \tv{f^3}{gQ}{6}.
\end{align}
These generators are also listed in Table~\ref{table:numgen},
arranged by index $m$ and order $\omega$.
Their explicit forms are easily obtained
by using \eqref{eq:isosubs} and \eqref{eq:abform}.
We obtain, for instance,
\begin{align}
\gen{1}{0}{0}{2}
 &=f
 &&=a_0
 &&=\frac{E_4}{12},\nn\\
\gen{0}{1}{0}{3}
 &=g
 &&=b_0
 &&=\frac{E_6}{216},\nn\\
\gen{1}{1}{2}{3}
 &=\tv{f}{g}{1}
 &&=\frac{a_0b_1}{3}
 &&=\frac{\Delta K}{36},\nn\\
\gen{2}{0}{4}{0}
 &=\tv{f}{f}{2}
 &&=2a_0a_2
 &&=\frac{6\Delta K^2-E_4E_6L+E_4^2M}{144},\nn\\
\gen{1}{1}{4}{1}
 &=\tv{f}{g}{2}
 &&=\frac{a_0b_2+3a_2b_0}{3}
 &&=\frac{-(E_6^2+576\Delta)L+E_4E_6M}{3456},\nn\\
\gen{0}{2}{4}{2}
 &=\tv{g}{g}{2}
 &&=\frac{2(3b_0b_2-b_1^2)}{9}
 &&=\frac{-12E_4\Delta K^2-E_4^2E_6L+E_6^2M}{93312}.
\end{align}
\begin{table}[t]
\begin{align*}
\begin{array}{|c|cccc|c|}
\hline
m\backslash\omega&0&1&2&3&\#\\ \hline
 0& -& -& f& g& 2\\
 2& -& -& -& \tv{f}{g}{1}& 1\\
 4& D=\tv{f}{f}{2}& \tv{f}{g}{2}&\ P=\tv{g}{g}{2}\ & -& 3\\
 6& -& \tv{f^2}{g}{3}& \tv{f}{P}{1}&\ Q=\tv{g}{P}{1}\ & 3\\
 8& \tv{f}{P}{2}& -& -& -& 1\\
10& -& \tv{f}{Q}{2}& -& -& 1\\
12&\ \tv{f^3}{g^2}{6},\ \cD=\tv{P}{P}{2}\ &\ \tv{f^2}{Q}{3}\ &-&-&3\\
14& -& -& -& -& 0\\
16& -& -& -& -& 0\\
18& \tv{f^3}{gQ}{6}& -& -& -& 1\\ \hline
\mbox{Tot.}&5&4&3&3&15\\
\hline
\end{array}
\end{align*}
\caption{Generators of $\Rtri_{*,*}$,
arranged by degree $m$ and order $\omega$.
The weight of generators is given by $k=3m+2\omega$.}
\label{table:numgen}
\end{table}

A rather trivial, but noteworthy observation is that
the 15 generators are all polynomials of $K,L,M,N$ over $M_*$.
This leads to the following proposition:
\begin{prop}\label{prop:RtriKLMN}
The ring of $D_4$ triality invariants
is contained in the polynomial ring $M_*[K,L,M,N]$.
In other words,
\begin{align}
\Rtri_{*,*}\subset \bbC[E_4,E_6,K,L,M,N].
\end{align}
\end{prop}
This proposition is stronger than Lemma~\ref{lem:RtriDinv} and
provides us with a simple
necessary condition for a function to be a $D_4$ triality invariant.
In addition to the algorithm~\ref{algo:D4},
this also serves as a convenient tool for efficiently calculating
physical quantities of the $N_\mathrm{f}=4$ theory.

\vspace{2ex}

\begin{center}
  {\bf Acknowledgments}
\end{center}

The author would like to thank Makoto Kawashima for discussions.
This work was supported in part by JSPS KAKENHI Grant Number 25K07326.

\vspace{2ex}


\appendix
\renewcommand{\theequation}{\Alph{section}.\arabic{equation}}

\section{Special functions}\label{app:funct}

The Jacobi theta functions are defined as
\begin{align}
\begin{aligned}
\varth_1(z,\tau)&:=
 i\sum_{n\in\bbZ}(-1)^n y^{n-1/2}q^{(n-1/2)^2/2},\\
\varth_2(z,\tau)&:=
  \sum_{n\in\bbZ}y^{n-1/2}q^{(n-1/2)^2/2},\\
\varth_3(z,\tau)&:=
  \sum_{n\in\bbZ}y^n q^{n^2/2},\\
\varth_4(z,\tau)&:=
  \sum_{n\in\bbZ}(-1)^n y^n q^{n^2/2},
\end{aligned}
\end{align}
where
\begin{align}
y=e^{2\pi i z},\qquad q=e^{2\pi i\tau}
\end{align}
and $z\in\bbC, \tau\in\bbH$.
The Dedekind eta function is defined as
\begin{align}
\eta(\tau):=q^{1/24}\prod_{n=1}^\infty(1-q^n).
\end{align}
The Eisenstein series are given by
\begin{align}
E_{2n}(\tau)
 =1-\frac{4n}{B_{2n}}\sum_{k=1}^{\infty}\frac{k^{2n-1}q^k}{1-q^k}
\label{eq:eisendef}
\end{align}
for $n\in\bbZ_{>0}$. The Bernoulli numbers $B_k$ are defined by
\begin{align}
\frac{x}{e^x-1}=\sum_{k=0}^\infty\frac{B_k}{k!}x^k.
\end{align}
We often abbreviate $\varth_k(0,\tau),\eta(\tau),\,E_{2n}(\tau)$
as $\varth_k,\eta,\,E_{2n}$ respectively.


\renewcommand{\section}{\subsection}
\renewcommand{\refname}

\end{document}